\documentclass[11pt]{amsart} 
 \usepackage[T1]{fontenc}
\usepackage{centernot}
\usepackage{mathtools}
\usepackage{stmaryrd}
\usepackage{tikz-cd}
\usepackage{enumitem}
\usepackage{bbold, adjustbox}
\usepackage{tikz}
\usetikzlibrary{tqft}
\makeatletter
\newcommand{\xMapsto}[2][]{\ext@arrow 0599{\Mapstofill@}{#1}{#2}}
\def\Mapstofill@{\arrowfill@{\Mapstochar\Relbar}\Relbar\Rightarrow}
\makeatother
\newcommand\abs[1]{\left|#1\right|} 
\newcommand\paren[1]{\left(#1\right)}
\newcommand\bra[1]{\left\{#1\right\}}
\newcommand{\norm}[1]{\left\lVert#1\right\rVert}
\newcommand{\brak}[1]{\left[#1\right]}
\usepackage[mathcal]{euscript}
\usepackage{amssymb, amsfonts, amsmath}
\usepackage{xypic}
\usepackage{amscd}
\usepackage{dutchcal}
\usepackage{amsthm}
\usepackage[margin=1.5in]{geometry}
\usepackage{xcolor}
\usepackage{comment}
\usepackage[style = numeric, doi=false]{biblatex}
\usepackage{hyperref}
\usepackage{xurl}
\hypersetup{breaklinks=true}
\renewbibmacro{in:}{}
\definecolor{winered}{rgb}{0.5,0,0}
 \usepackage{hyperref}
 \hypersetup{citecolor = winered, linkcolor = winered, menucolor = winered, colorlinks = true}
\usepackage{ circuitikz }
\usepackage{algorithm}

\newtheorem{theorem}{Theorem}
\newtheorem{lemma}[theorem]{Lemma}
\newtheorem{corollary}[theorem]{Corollary}

\newtheorem{problem}{Problem}
\newtheorem{theorem-definition}[theorem]{Theorem-Definition}

\theoremstyle{definition}

\newtheorem{remark}[theorem]{Remark}
\newtheorem*{theorem*}{Theorem}

\numberwithin{equation}{section} \numberwithin{figure}{section}

\numberwithin{equation}{section}

\newcommand{\E}{\E_{\infty}}

\usepackage{scalerel,amssymb}

\usepackage{tabularx} 
\usepackage{ragged2e} 
\newcolumntype{L}{>{\RaggedRight}X} 
\usepackage{graphicx}
\usepackage{float}

\title{Direct Imaging Methods for Inverse Acoustic Obstacle Scattering}
\author{General Ozochiawaeze}
\date{}
\begin{document}
\maketitle
\begin{abstract}
 Direct imaging methods recover the presence, position, and shape of the unknown obstacles in time-harmonic inverse scattering without \textit{a priori} knowledge of either the physical properties or the number of disconnected components of the scatterer, i.e., on the boundary condition. However, most of these methods require multi-static data and only obtain partial information about the obstacle. These qualitative methods are based on constructing indicator functions defined on the domain of interest, which help determine whether a spatial point or point source lies inside or outside the scatterer. This paper explains the main themes of each of these methods, with emphasis on highlighting the advantages and limitations of each scheme. Additionally, we will classify each method and describe how some of these methods are closely related to each other.
\end{abstract}
\section*{Introduction and Preliminaries}
In time-harmonic inverse obstacle scattering, the primary goal is to extract information about unknown objects from the scattered wave-fields far away from the target. Inverse obstacle scattering has applications in non-destructive testing, geological exploration, radar, medical, and seismic imaging, and deep-sea exploration. In order to fully recover an obstacle's information, one has to solve a fully nonlinear inverse problem, which is often solved by an iterative or decomposition method. However, iterative and decomposition methods often carry high computational cost and rely on some \textit{a priori} information for obtaining initial approximations. In many practical situations, it suffices to provide basic information, such as how many scatterers are present, their location, shape, and size. For these cases, \textit{direct imaging methods}, also known as \textit{qualitative methods}, are preferred. More precisely, these direct imaging methods are based on choosing an appropiate indicator function $g$ defined on the domain of interest $D$ such that its value $g(z)$ decides whether that point $z$ lies inside or outside the scatterer $D$. Since each spatial point in the domain is sampled, direct imaging methods are often referred to as \textit{sampling methods}. The definition of an indicator function usually exploits only the linearly mapped portion of the forward problem, requiring less computational cost. Furthermore, the implementation of sampling methods do not require the boundary condition to be known in advance. However, common drawbacks include needing multi-static data, i.e., knowledge of the far field pattern for a large number of incident fields. Moreover, only partial qualitative information about the scatterer is obtained.\\

This paper gives a comparison of the most well-known qualitative inversion schemes. We will identify the main themes underpinning these shape reconstruction schemes and discuss the mathematical aspects and justification of these methods. Additionally, we will emphasize the advantages and disadvantages of each inversion scheme.\\

In this paper, we will be chiefly concerned with the \textit{acoustic sound soft obstacle scattering problem}. The propagation of time-harmonic acoustic fields in on a homogeneous obstacle is governed by the Helmholtz equation:
\begin{align}
    \Delta u+k^2u&= 0.
\end{align}
We will assume $k>0$, i.e., the scatterer is non-absorbing.\\

Let us now consider a bounded obstacle $D\subset \mathbb{R}^2$ of class $C^2$ such that the exterior of this obstacle $\mathbb{R}^2\setminus\overline{D}$ is connected. Given an incident field $u^{i}$, which is assumed to be a solution of (0.1) in $\mathbb{R}^2$, the presence of $D$ will yield a scattered field $u^{s}$. The sound soft problem consists of solving (0.1) with a Dirichlet boundary condition for the total field on $\partial D$:
\begin{align}
    u^{s}+u^{i}&=0 \quad \text{on } \partial D.
\end{align}
To complete the formulation of the problem, we include the Sommerfield radiation condition imposed on the scattered field $u^{s}$ specifying that there are no incoming waves from infinity:
\begin{align}
    \lim_{|x|\to \infty}\sqrt{|x|}\left(\frac{\partial u^{s}}{\partial |x|}-iku^{s}\right)&=0, \quad x\in \mathbb{R}^2\setminus\overline{D},
\end{align}
uniformly for all directions $\Hat{x}=\frac{x}{|x|}\in \Omega$, where $\Omega=\{x\in\mathbb{R}^2\,:\, |x|=1\}$ is the unit circle. This condition ensures uniqueness of the solution for the exterior problem. The problem posed by equations (0.1)-(0.3) is known as the \textit{direct obstacle scattering problem}. The well-posedness of the problem is well-understood \cite{arens2003linear}. \\

The \textit{inverse obstacle reconstruction problem} refers to determining the shape of the obstacle $D$ from knowledge of $u^{s}$ for incident fields $u^{i}$. To be more precise, we recall that the scattered field $u^{s}$ has a particular asymptotic behavior as $|x|\to \infty$, namely
\begin{align}
    u^{s}&=\frac{\exp{ikr}}{\sqrt{r}}u_{\infty}(\Hat{x},d)+\mathcal{O}\left(r^{-3/2}\right),
\end{align}
as $r\to \infty$ where $r=|x|$, and $\Hat{x}=\frac{x}{r}$, $k$ is fixed and $u_{\infty}$ is the far field pattern of the scattered field. The far field pattern represents the amplitude of the scattered field when the scattered wave is far away from the scatterer. The far field pattern can be expressed as
\begin{align}
    u_{\infty}(\Hat{x},d)&=\frac{\exp{i\frac{\pi}{4}}}{\sqrt{8\pi k}}\int_{\partial D}\left(u^{s}\frac{\partial}{\partial \nu}\exp{-ik\Hat{x}\cdot y}-\frac{\partial u^{s}}{\partial \nu}\exp{-ik\Hat{x}\cdot y}\right)\,ds(y).
\end{align}
We will be particularly concerned with scattered fields and far field patterns arising from incident plane waves, i.e., incident fields of the form
\begin{align*}
    u^{i}(x,d)&\coloneqq \exp(ikd\cdot x), \quad d\in \Omega.
\end{align*}
Sampling methods are methods used to solve the following inverse obstacle reconstruction problem.
\begin{problem}
    Given the knowledge of $u_{\infty}(\cdot,\cdot)$ on $\Omega\times \Omega$, determine the shape of the obstacle $D$.
\end{problem}
It is well-known that problem 1 is uniquely solvable for the sound-soft obstacle \cite{colton1998inverse}. Fundamental to the different approaches we discuss to solving problem 1 is the linear integral operator $F: L^2(\Omega)\to L^2(\Omega)$, given by
\begin{align}
    F g(\Hat{x})&\coloneqq \int_{\Omega}u_{\infty}(\Hat{x},d)g(d)\,ds(d), \quad \Hat{x}\in\Omega,
\end{align}
called the far field operator. The far field operator $F$ can also be thought of as the measurement operator of the inverse problem. The first direct imaging method we will cover is the \textit{linear sampling method}, which attempts to solve the far field equation for kernel $g$ with a certain right-hand side for each point in $x\in\mathbb{R}^2$. The numerical observation is that the approximate solution of this equation will have a large norm outside and close to $\partial D$. Hence, by plotting the norm of the solution it is possible to obtain an image of the obstacle.
\section{Linear Sampling Method}
The first direct imaging methods we discuss are based on testing the range of an operator to obtain shape reconstructions. Roughly speaking, the linear sampling method (LSM) is an algorithm that will determine the shape of the obstacle $D$ by finding approximate solutions to the far-field equation
\begin{align*}
    \int_0^{2\pi}u_{\infty}(\theta,\varphi)g(\varphi)\,d\varphi&=\gamma\exp{(-ikr_z\cos{(\theta-\theta_z)})},
\end{align*}
where $(r_z,\theta_z)$ are the polar coordinates of a point $z\in \mathbb{R}^2$ and
\begin{align*}
    \gamma&=\frac{e^{i\pi/4}}{\sqrt{8\pi k}},
\end{align*}
or, in simpler notation,
\begin{align*}
    (Fg_z)(\Hat{x})=\Phi_{\infty}(\Hat{x},z),
\end{align*}
where $\Hat{x}=(\cos{\theta}, \sin{\theta})$ and
\begin{align*}
    \Phi_{\infty}(\Hat{x},z)&=\gamma e^{-ik\Hat{x}\cdot z}
\end{align*}
is the far field pattern of the fundamental solution to the Helmholtz equation. The linear sampling method is based on numerically determining the function $g_z$, which will serve as an indicator function that will help determine the points $z$ that lie in the scatterer $D$. However, the numerical procedure we will describe is rather ad-hoc since in general the far field equation has no actual solution even in the case of ``noise-free" data $u_{\infty}$.\\

We begin our discussion of the linear sampling method by considering the fact that the direct scattering problem (0.1)-(0.3) is a special case of the exterior boundary value problem.
\begin{problem}
    Given $f\in H^{\frac{1}{2}}(\partial D)$, compute the far field pattern $v_{\infty}$ of the weak solution $v$ of (0.1) satisfying (0.3), and $v|_{\partial D}=f$.
\end{problem}
Hence, we are considering weak solutions to the exterior boundary value problem. We can then define the boundary operator $B:H^{\frac{1}{2}}(\partial D)\to L^2([0,2\pi])$, which is defined as the linear operator mapping $f$ into the far field pattern $u_{\infty}$ corresponding to problem 2, i.e., $Bf=v_{\infty}$. $B$ is also referred to as the solution operator of problem 2. The operators $F$ and $B$ are closely related: consider the incident field
\begin{align*}
    v_g(x)&\coloneqq \int_{\Omega}\exp({ikd\cdot x})g(d)\,ds(d), \quad x\in \mathbb{R}^2.
\end{align*}
By linearity of the direct obstacle scattering problem, $Fg$ is the far field pattern of the scattered field arising from the incident field $v_g$, i.e., $Fg=B(-v_g|_{\partial D})$. Introducing the Herglotz operator $H: L^2([0,2\pi])\to H^{1/2}(\partial D)$ by
\begin{align*}
    Hg(x)&\coloneqq v_g|_{\partial D},
\end{align*}
we can rewrite the previous relation as
\begin{align}
    F&=-BH.
\end{align}
That is, we obtain the following commutative diagram
\[
\begin{tikzcd}
 & H^{\frac{1}{2}}(\partial D) \arrow{dr}{B} \\
L^2([0,2\pi]) \arrow{ur}{H} \arrow{rr}{F} && L^2([0,2\pi])
\end{tikzcd}
\]
The function $v_g$ is the Herglotz wavefunction with density $g$. The following theorem sums up the important properties of the operators $F$, $B$, and $H$.
\begin{theorem}
    The operator $H$ has dense range and the operator $B$ is injective. The operators $F$ and $B$ are both compact with $F$ being normal for the sound-soft problem. Suppose the Herglotz operator $H$ is additionally injective, i.e., that there is no Herglotz wavefunction $v_g$ that is a Dirichlet eigenfunction of the negative Laplacian for the eigenvalue $k^2$ in $D$. Then 
    \begin{itemize}
        \item $F$ is injective.
        \item Both $F$ and $B$ have dense range.
    \end{itemize}
\end{theorem}
For a proof of this theorem, see \cite{cakoni2014qualitative}. Another important fact about the boundary operator $B$ is that this operator characterizes the unknown obstacle completely. Namely, the far field pattern of a point source located at $z\in \mathbb{R}^2$ is obtainable from the asymptotic behavior of the fundamental solution as $|x|$ tends to infinity; this is denoted $\Phi_{\infty}$ as previously discussed. We can then show, using Rellich's lemma, that 
\begin{align}
    z\in D\iff \Phi_{\infty}(\cdot,z)\in B(H^{\frac{1}{2}}(\partial D)),
\end{align}
i.e., the range of the operator $B$ (see \cite{cakoni2014qualitative} and \cite{colton1998inverse}). 
\begin{proof}[Proof of (1.2)]
     \ \\ $\Rrightarrow$ First let $z\in D$. Then define
     $v(x)\coloneqq \Phi(x,z)=\frac{i}{4}H_0^{(1)}(k|x-z|), x\notin D$ and $f(x)\coloneqq v|_{\partial D}$. Then we know $f\in H^{1/2}(\partial D)$ by the trace theorem and the far field pattern of $v$ is given by $v_{\infty}=\gamma e^{-ik\Hat{x}\cdot z}, \Hat{x}\in [0,2\pi]$, which coincides with the far field pattern of the fundamental solution to the Helmholtz equation, namely $v_{\infty}(\Hat{x})=\Phi_{\infty}(\cdot, z)$. Hence, 
    \begin{align*}
        Bf&=v_{\infty}=\Phi_{\infty}(\cdot, z)\implies \Phi_{\infty}\in B(H^{1/2}(\partial D)).
    \end{align*}
    $\Lleftarrow$ Suppose $z\notin D$. We proceed by contradiction and suppose additionally that $\Phi_{\infty}(\cdot,z)\in B(H^{1/2}(\partial D))$. We divide into two cases
    \begin{enumerate}
        \item Suppose $z\in \mathbb{R}^2\setminus\overline{D}$. Then since $\Phi_{\infty}(\cdot,z)\in B(H^{1/2}(\partial D))$, there is $f\in H^{1/2}(\partial D)$ such that $Bf=\Phi_{\infty}(\cdot,z)$. Let $v$ be a solution to the exterior Dirichlet boundary value problem, i.e., $v$ satisfies
        \begin{align*}
            \Delta v+k^2v&=0 \quad \text{outside }D\\
            v|_{\partial D}&=f\\
            \lim_{|x|\to \infty}\sqrt{|x|}\paren{\frac{\partial v}{\partial |x|}-ikv}&=0.
        \end{align*}
        Then let $v_{\infty}=Bf$ denote the far field pattern corresponding to the scattered field solution $v$. Then by Rellich's lemma and the analyticity of the solution $v$ with trace $v|_{\partial D}=f$, $v$ must coincide with the fundamental solution $\Phi(\cdot,z)$ in $\mathbb{R}^2\setminus\paren{\overline{D}\cup \{z\}}$ since $\Phi_{\infty}$ coincides with $v_{\infty}$. But if $z\in \mathbb{R}^2\setminus\overline{D}$, then we obtain a contradiction to the fact that $v$ is analytic in $\mathbb{R}^2\setminus\overline{D}$ and $\Phi(\cdot,z)$ has a singularity at $x=z$, despite $v=\Phi(x,z)$. Hence $z\notin \mathbb{R}^2\setminus\overline{D}$.
        \item Suppose instead $z\in \partial D$. Now $\Phi(x,z)=f(x)$ for $x\in\partial D,\, x\neq z$. Since $f\in H^{1/2}(\partial D)$. $\Phi(x,z)|_{\partial D}\in H^{1/2}(\partial D)$. But this is impossible as $\nabla \Phi(x,z)=\mathcal{O}\paren{\frac{1}{|x-z|^2}}$ as $x\to z$ (this is an asymptotic result on the differentiability of the Hankel function of the first kind of order zero). So $\nabla\Phi(\cdot,z)$ is neither 
 in $L_{loc}^2(\mathbb{R}^2\setminus\overline{D})$ nor in $L^2(D)$, so we obtain a contradiction. Hence, $z\notin \partial D$ either, meaning $z\in D$.
    \end{enumerate}
\end{proof}
LSM then combines (1.1) and (1.2) to approximately solve the far-field equation. The norm of the approximate solution $||g_z||_{L^2([0,2\pi])}$ should then blow up as $z\in D$ approaches the boundary of the scatterer $\partial D$ and it should continue to be large for $z\notin \overline{D}$. The following theorem provides a mathematical basis of the linear sampling method.
\begin{theorem}
    Let $u_{\infty}$ be the far field pattern associated to the scattering problem (0.1)-(0.3) with the associated far field operator $F$. Assume that $k^2$ is not a Dirichlet eigenvalue of the negative Laplacian in $D$. Then the following holds
    \begin{enumerate}
        \item If $z\in D$ then for every $\epsilon>0$ there is a solution $g_z^{\alpha}\in L^2([0,2\pi])$ of the inequality 
        \begin{align*}
            ||Fg_z^{\alpha}-\Phi_{\infty}(\cdot,z)|| &\,<\,\epsilon,
        \end{align*}
        such that
        \begin{align*}
            \lim_{z\to \partial D}||g_z^{\alpha}||_{L^2([0,2\pi])}&=\infty \qquad\text{and } \lim_{z\to\partial D}||v_{g_z^{\alpha}}||_{H^1(D)}=\infty.
        \end{align*}
        \item If $z\notin \overline{D}$ then for every $\epsilon>0$ and $\delta>0$ there exists a solution $g_z^{\alpha,\delta}\in L^2([0,2\pi])$ of the inequality
              \begin{align*}
            ||Fg_z^{\alpha,\delta}-\Phi_{\infty}(\cdot,z)|| &\,<\,\epsilon+\delta,
        \end{align*}
        such that 
        \begin{align*}
            \lim_{\delta\to 0}||g_z^{\alpha,\delta}||_{L^2([0,2\pi])}&=\infty \qquad\text{and } \lim_{\delta\to 0}||v_{g_z^{\alpha,\delta}}||_{H^1(D)}=\infty.
        \end{align*}
    \end{enumerate}
\end{theorem}
The main problem with this theorem is that there is no guarantee that a numerical algorithm for solving the inequalities, that is, a regularized method to solve the integral equation of the first kind, will actually obtain these densities $g_z^{\alpha}$ and $g_z^{\alpha,\delta}$, respectively. This is because the right-hand side of the far-field equation is almost never an element of the range of the operator $F$ but rather $B$ \cite{arens2003linear}. Nevertheless, numerically the method has proven effective for a large class of obstacle scattering problems.
\begin{proof}
    Assume $z\in D$. Then we know then $\Phi_{\infty}(\cdot,z)\in B(H^{1/2}(\partial D))$. So there exists $f_z\in H^{1/2}(\partial D)$ such that $Bf_z=-\Phi_{\infty}(\cdot,z)$. Then since the operator $H$ is bounded, injective, and has dense range, we see that for a given $\epsilon>0$ there exists a Herglotz wave function with kernel $g_z^{\alpha}\in L^2([0,2\pi])$ such that
    \begin{align*}
        ||Hg_z^{\alpha}-f_z||_{H^{1/2}(\partial D)}<\frac{\epsilon}{||B||}.
    \end{align*}
    Consequently,
    \begin{align*}
        ||BHg_z^{\alpha}-Bf_z||_{L^2([0,2\pi])}<\epsilon.
    \end{align*}
    Note that $F=-BH$, so we have that
    \begin{align*}
        ||Fg_z^{\alpha}-\Phi_{\infty}||_{L^2([0,2\pi])}<\epsilon.
    \end{align*}
    Next, notice that the Herglotz wave function $v_g^{\alpha}$ with kernel $g_z^{\alpha}$ satisfies the interior Dirichlet problem with $h\coloneqq -Hg_z^{\alpha}$, namely, we have
    \begin{align*}
        \Delta v_g^{\alpha}+k^2v_g^{\alpha}&=0 \quad \text{in }D\\
        v_g^{\alpha}|_{\partial D}&=h.
    \end{align*}
    Hence, we see $v_g^{\alpha}\to v_z$ of the interior Dirichlet problem.\\
    Now assume $z\notin D$ and assume to the contrary that there is a sequence $\{\epsilon_n\}\to 0$ and corresponding $g_n$ satisfying $||Fg_n-\Phi_{\infty}(\cdot,z)||_{L^2([0,2\pi])}<\epsilon_n$ such that the $||v_n||_{H^1(D)}$ remain bounded where $v_n\coloneqq v_{g_n}$ refers to the Herglotz wave function with kernel $g_n$. Then by the trace theorem $||Hg_n||_{H^{1/2}(\partial D)}$ also remain bounded. Without loss of generality, we may assume weak convergence of $Hg_n\to h\in H^{1/2}(\partial D)$ as $n\to\infty$. Since $B$ is a bounded linear operator, we have $BHg_n\to Bh$ in $L^2([0,2\pi])$. But then $BHg_n\to -\Phi_{\infty}(\cdot,z)$ which means we have $Bh=-\Phi_{\infty}(\cdot,z)$, a contradiction since we assumed $z\notin D$. Hence the second statement of the theorem remains true.
\end{proof}
The procedure of LSM can be summarized as follows:
\begin{remark}[Linear Sampling Method]
The linear sampling method provides a scheme for the reconstruction of $D$ by a regularized solution of the integral equation $Fg_z=\Phi_{\infty}(\cdot,z)$ from the scattering data $u_{\infty}(\Hat{x},d)$ for $\Hat{x},d\in \Omega$, where $\Omega$ is the unit circle or sphere in dimensions $2$ and $3$ respectively.
\begin{enumerate}
    \item Select a grid of ``sampling points" in a region known to contain $D$. Namely, choose a sampling grid $\mathcal{G}$, a regularization parameter $\alpha>0$ and a cut-off constant $c_0$.
    \item Use Tikhonov-Morozov regularization to compute an approximation $g_z$ to the far-field equation for each $z$ in the foregoing grid $\mathcal{G}$. (It is possible to use other regularizations, but Tikhonov-Morozov is most commonly used).
    \item Calculate a reconstruction $M$ for $D$ by
    \begin{align*}
        M&\coloneqq \bra{z\in\mathcal{G}\,:\,\norm{g_{z,\alpha}}\leq c_0}.
    \end{align*}
    The choice of $c_0$ here is heuristic but improves when the frequency becomes higher \cite{cakoni2014qualitative}.
\end{enumerate}
\end{remark}
Though we have justification for the case when $z\notin D$, one major issue with the theoretical framework of LSM is that we have no guaranteed way of indicating when $z\in D$. Hence, we have no guarantee that solving the regularized far-field equation will actually obtain the desired densities. This disadvantage of LSM motivates the factorization method as a similar alternative that fixes this issue, assuming more structural assumptions. Nevertheless, LSM has proven to be effective for many different obstacle reconstruction problems.\\

Due to the compactness of the operator $F$, the far-field equation is ill-posed. Hence, a stable solution in the generic sampling point $z$ requires a regularization. Usually, this is done using Tikhonov regularization, thus obtaining the equation
\begin{align}
    (\alpha I+F^{*}F)g_{z,\alpha}&=F^{*}\Phi_{\infty}(\cdot,z),
\end{align}
where $\alpha>0$ is the regularization parameter. For the sound-soft Dirichlet problem under consideration, the far field equation is normal. Since $F$ is additionally compact, by spectral theory there exist eigenvalues $\lambda_n\in \mathbb{C}$ of $F$ and corresponding eigenfunctions $g_n\in L^2([0,2\pi])$, $n\in \mathbb{N}$, such that the set $\{g_n\}$ forms a complete orthonormal system in $L^2([0,2\pi])$. The operator $F$ then obtains an eigenexpansion
\begin{align}
    Fg&=\sum_{n=1}^{\infty}\lambda_n(g,g_n)g_n,\quad g\in L^2([0,2\pi]).
\end{align}
Using the eigensystem of $F$, the unique solution of (1.3) can be written in the form
\begin{align}
    g_{z,\alpha}&=\sum_{n=1}^{\infty}\frac{\overline{\lambda_n}}{\alpha+|\lambda_n|^2}(\Phi_{\infty}(\cdot,z),g_n)g_n,
\end{align}
where evaluation is computationally straightforward as it requires a single evaluation of the singular value decomposition of $F$.
\section{Factorization Method}
The factorization method can be viewed as a refinement of the linear sampling method; it attempts to rectify the theoretical issues LSM has. Both methods try to determine the support of the scatterer by deciding whether a point $z$ in space is inside or outside the scatterer. However, in the factorization method, when $F$ is normal, which is the case for the sound-soft obstacle scattering problem, we study the equation
\begin{align}
    (F^{*}F)^{\frac{1}{4}}g_z&=\Phi_{\infty}(\cdot,z),
\end{align}
replacing the far field operator $F$. One can show that (2.1) has a solution if and only if $z\in D$. This method relies on the factorization of $F$. \\
We first consider one of the main theoretical foundations of the factorization method given by the following optimization theorem.
\begin{theorem}
    Let $X$ and $H$ be Hilbert spaces with inner products $(\cdot,\dot)$, let $X^{*}$ be the dual space of $X$ and assume that $F:H\to H$, $A:X\to H$, and $T:X^{*}\to X$ be bounded linear operators that satisfy
    \begin{align}
        F&=ATA^{*}
    \end{align}
    where $A^{*}:H\to X^{*}$ is the antilinear adjoint of $A$ defined by
    \begin{align}
        \langle \varphi,A^{*}g\rangle&=\paren{A\varphi,g},\quad g\in H,\varphi\in X,
    \end{align}
    in terms of the bilinear duality pairing of $X$ and $X^{*}$. Assume further that $T$ is coercive, i.e., 
    \begin{align}
        \abs{\langle Tf,f\rangle}&\geq c||f||_{X^{*}}^2
    \end{align}
    for all $f\in A^{*}(H)$ for some $c>0$. Then for any $g\in H$ with $g\neq 0$ we have that 
    \begin{align}
        g\in A(X)&\iff \inf\bra{\abs{\paren{F\psi,\psi}}\,:\,\psi\in H,\paren{g,\psi}=1}>0.
    \end{align}
\end{theorem}
\begin{proof}
  Note that by assumptions (2.2) and (2.4), we obtain
    \begin{align*}
        \abs{\paren{F\psi,\psi}}&=\abs{\langle TA^{*}\psi,A^{*}\psi\rangle}\geq c\norm{A^{*}\psi}_{X^{*}}^2
    \end{align*}
    for all $\psi\in H$. Now assume $g=A\varphi$ for all $\varphi\in X$ and $g\neq 0$. Then for each $\psi\in H$ with $\paren{g,\psi}=1$, we obtain
    \begin{align*}
        c&=c\abs{\paren{g,\psi}}^2=c\abs{\paren{A\varphi,\psi}}^2=c\abs{\langle \varphi,A^{*}\psi\rangle}^2\leq c\norm{\varphi}_X^2\norm{A^{*}\psi}_{X^{*}}^2\leq \norm{\varphi}_X^2\abs{\paren{F\psi,\psi}},
    \end{align*}
    as desired.\\
    Conversely assume $g\neq A(X)$. We define $V\coloneqq [\text{span}\{g\}]^{\perp}$ and show $A^{*}(V)$ is dense in $\overline{A^{*}(H)}$. We can identify $X=J(X^{*})$, where $J$ is the antilinear isomorphism from the Riesz representation theorem given by
    \begin{align*}
        \langle \varphi,f\rangle&=\paren{\varphi, Jf},\quad \varphi\in X,f\in X^{*}.
    \end{align*}
    So $JA^{*}:H\to X$ is the adjoint of $A:X\to H$ and so it suffices to show $JA^{*}(V)$ is dense in $\overline{JA^{*}(H)}$. Let $\varphi=\displaystyle\lim_{n\to\infty}JA^{*}\psi_n$ with $\psi_n\in H$ be orthogonal to $JA^{*}(V)$. Then
    \begin{align*}
        (A\varphi,\psi)&=(\varphi,JA^{*}\psi)=0
    \end{align*}
    for all $\psi\in V$, implying $A\varphi\in V^{\perp}=\text{span}\{g\}$. But $g\neq A(X)$, so $A\varphi=0$. Then
    \begin{align*}
\norm{\varphi}^2&=\lim_{n\to\infty}\paren{\varphi,JA^{*}\psi_n}=\lim_{n\to\infty}\paren{A\varphi,\psi_n}=0.
    \end{align*}
    Hence $JA^{*}(V)$ is dense in $\overline{JA^{*}(H)}$. Now we can choose a sequence $\bra{\widetilde{\psi_n}}$ in $V$ so that
    \begin{align*}
        A^{*}\widetilde{\psi_n}\to -\frac{1}{\norm{g}^2}A^{*}g,\quad n\to\infty.
    \end{align*}
    Setting $\psi_n\coloneqq \widetilde{\psi_n}+\frac{1}{\norm{g}^2}g$, we have $(g,\psi_n)=1$ for all $n$ and $A^{*}\psi_n\to 0$ as $n\to\infty$. But observe that 
    \begin{align*}
        \abs{\paren{F\psi_n,\psi_n}}&\leq \norm{T}\norm{A^{*}\psi_n}_{X^{*}}^2\to 0,\quad n\to\infty,
    \end{align*}
    hence the result follows by contraposition.
\end{proof}
\noindent We will apply this theorem to the far-field operator $F$. We will choose the spaces $H=L^2([0,2\pi])$ and $X=H^{1/2}(\partial D)$. Then the assumptions of the theorem will hold for the operators $A=B:H^{1/2}(\partial D)\to L^2([0,2\pi])$, i.e., the data-to-solution operator and $T=-S^{*}:H^{-1/2}(\partial D)\to H^{1/2}(\partial D)$, the adjoint of the single-layer potential operator. \\
Recall from the linear sampling method that $F$ has the factorization $F=-BH$. The operators $B$ and $H$ are also related to the single-layer operator $S: H^{-1/2}(\partial D)\to H^{1/2}(\partial D)$ on $\partial D$, defined by
\begin{align*}
    S\psi(z)&=\int_{\partial D}\Phi(x,z)\psi(z)\,ds(z),\quad \psi \in H^{-1/2}(\partial D),
\end{align*}
with 
\begin{align*}
    \Phi(x,z)&\coloneqq \frac{i}{4}H_0^{(1)}(k|x-z|),
\end{align*}
the fundamental solution of the Helmholtz equation with $H_0^{(1)}$ denoting the Hankel function of the first kind. We can then verify that we have the relation
\begin{align}
    BS\psi&=H^{*}\psi, \quad \psi\in H^{-1/2}(\partial D).
\end{align}
Hence, the far field operator has the following factorization
\begin{align}
    F=-BS^{*}B^{*},
\end{align}
i.e., the following diagram commutes:
\[
\begin{tikzcd}[row sep=large, column sep = large]
H^{\frac{1}{2}}(\partial D) \arrow[leftarrow]{r}{S^{*}} \arrow[swap]{d}{B} & H^{-\frac{1}{2}}(\partial D) \arrow[<-]{d}{B^{*}} \\
L^2([0,2\pi])  \arrow{r}{F} & L^2([0,2\pi])
\end{tikzcd}
\]
Since the far field operator $F$ is normal and compact for the sound-soft obstacle, there exists a complete set of orthogonal eigenfunctions $g_n\in L^2([0,2\pi])$ with corresponding eigenvalues $\lambda_n\in L^2(\mathbb{C}), n\in \mathbb{N}$. The spectral theorem yields the eigenexpansion of $F$ given by (1.4). As a conclusion, the far field operator $F$ has a second factorization of the form
\begin{align}
    F&=(F^{*}F)^{1/4}J(F^{*}F)^{1/4},
\end{align}
where the operator $(F^{*}F)^{1/4}: L^2([0,2\pi])\to L^2([0,2\pi])$ is given by
\begin{align*}
    (F^{*}F)^{1/4}g&=\sum_{n=1}^{\infty}\sqrt{|\lambda_n|}(g,g_n)\,g_n, \quad g\in L^2([0,2\pi]),
\end{align*}
and $J: L^2([0,2\pi])\to L^2([0,2\pi])$ of $F$ is given by
\begin{align*}
    Jg&=\sum_{n=1}^{\infty}\frac{\lambda_n}{|\lambda_n|}(g,g_n)\,g_n, \quad g\in L^2([0,2\pi]).
\end{align*}
This factorization method is based on the following observation, namely,
\begin{align*}
    (Fg,g)&= (S^{*}B^{*}g, Bg)\\
        &\leq ||S^{*}||\cdot ||B^{*}Bg||\\
        &=||S||\cdot||Bg||^2.
\end{align*}
If $F$ had a square root, i.e., 
\begin{align*}
    (Fg,g)&=||F^{1/2}g||^2\geq \alpha ||Bg||, \quad \alpha>0,
\end{align*}
then $\mathcal{R}(B)=\mathcal{R}(F^{1/2})$, where $\mathcal{R}$ denotes the range of the linear operator. However, $F$ is not a positive operator, but we can form a unique positive operator of $F$, namely, $|F|=(F^{*}F)^{1/2}$, since $F$ is normal. Then $\mathcal{R}(B)=\mathcal{R}(|F|^{1/2})$, and since we know $\Phi_{\infty}(\cdot,z)\in \mathcal{R}(B)$ if and only if $z\in D$, we have 
\begin{align}
    \Phi_{\infty}(\cdot,z)\in \mathcal{R}(|F|^{1/2}) \iff  z\in D.
\end{align}
Hence, what is advantageous about the factorization method is we obtain a nice result whose dependence lies only on the far field operator $F$. However, we incorporated a few additional assumptions that were not present in LSM:
\begin{enumerate}
    \item $S$ must be a bounded, coercive operator.
    \item $S$ must have the form $S=S_0+C$ for some compact operator $C$ and some self-adjoint operator $S_0$ which is coercive on $\mathcal{R}(B^{*})$.
    \item The compact, normal operator $F$ must additionally be injective.
\end{enumerate}
These three assumptions must hold for the single layer operator for the ranges of $B$ and $|F|^{1/2}$ to coincide. Fortunately, we have the following lemma.
\begin{lemma}
    Assume $k^2$ is not a Dirichlet eigenvalue of $-\Delta$ in $D$. Then the following holds.
    \begin{enumerate}
        \item $\operatorname{Im}\langle\varphi, S\varphi\rangle\neq 0$ for all $\varphi\in H^{-1/2}(\partial D)$ with $\varphi\neq 0$.
        \item Let $S_i$ be the single layer operator of $S$ corresponding to the wave number $k=i$. Then the operator $S_i$ is self-adjoint and coercive as an operator from $H^{-1/2}(\partial D)\to H^{1/2}(\partial D)$.
        \item The difference $S-S_i$ is compact from $H^{-1/2}(\partial D)\to H^{1/2}(\partial D)$.
    \end{enumerate}
\end{lemma}
\noindent A proof of this lemma is provided in \cite{kirsch2007factorization} and \cite{colton1998inverse}. This lemma establishes the choice of $T=-S^{*}$.
\begin{corollary}
    Let $F$ be the far field operator and assume $k^2$ is not a Dirichlet eigenvalue of the negative Laplacian in $D$. Then \begin{align*}
        z\in D&\iff \inf\bra{\abs{(F\psi,\psi)}\,:\, \psi\in L^2([0,2\pi]),\,\paren{\psi,\Phi_{\infty}(\cdot,z)}=1}>0.
    \end{align*}
\end{corollary}
\noindent So we have a variational method for determining $D$ from knowledge of the far field pattern for all incident and observation directions, though this approach is especially time-consuming in practice since it demands solving a minimization problem for every sampling point $z$. The factorization method uses a more efficient approach in the spirit of the LSM, involving range identity characterization. \\
 
\noindent The factorization method essentially characterizes the range of the boundary operator $B$ (and therefore the obstacle $D$) in terms of the measurement operator $F$, i.e., in terms of the singular system of $F$.
\begin{theorem}
    Let $X,H$ be Hilbert spaces and let the bounded operators $F,T,A$ satisfy the assumptions of Theorem 3 with $T=S^{*}$ additionally satisfying the few additional assumptions (1)-(3). In addition let $F:H\to H$ be compact, injective and assume that $I+i\gamma F$ is unitary for some $\gamma>0$. Then the ranges of $A(X)$ and $|F|^{1/2}(H)$ coincide.
\end{theorem}
\begin{proof}
    Since $I+i\gamma F$ is unitary, $F$ is normal. Therefore, the spectral theorem for compact, normal operators implies that there exists a complete orthonormal set of eigenelements $\psi_n\in H$ with corresponding eigenvalues $\lambda_n, n=1,2,\dots$. So we obtain the following expansion for $F$ applied to $\psi\in H$:
    \begin{align*}
        F\psi&=\sum_{n=1}^{\infty}\lambda_n\paren{\psi,\psi_n}\psi_n,\quad \psi\in H.
    \end{align*}
    From this, we note what we observed earlier, namely $F$ has a second factorization of the form
    \begin{align*}
        F&=|F|^{1/2}J|F|^{1/2},
    \end{align*}
    where the operator $|F|^{1/2}=(F^{*}F)^{!/4}$ is given by
    \begin{align*}
|F|^{1/2}\psi&=\sum_{n=1}^{\infty}\sqrt{|\lambda_n|}\paren{\psi,\psi_n}\psi_n,\quad \psi\in H,
    \end{align*}
    and $J: H\to H$ is given by
    \begin{align*}
        J\psi&=\sum_{n=1}^{\infty}\frac{\lambda_n}{\abs{\lambda_n}}\paren{\psi,\psi_n}\psi_n,\quad \psi\in H.
    \end{align*}
    We need only verify that $J$ in fact satisfies the coercivity property (2.4) in Theorem 3, and the result holds. Since the operator $I+i\gamma F$ is unitary the eigenvalues $\lambda_n$ lie on the circle of radius $r\coloneqq 1/\gamma$ and center $ri$. We set $s_n\coloneqq \displaystyle\frac{\lambda_n}{\abs{\lambda_n}}$. From $\abs{\lambda_n-ri}=r$ and the only accumulation point $\lambda_n\to 0,n\to\infty$, we conclude only $1$ or $-1$ are the only possible accumulation points of the sequence $(s_n)$. However, we will check that the only accumulation point is $1$. If $1$ is the only accumulation point, then we can write $s_n=e^{it_n}$, where $0\leq t_n\leq \pi-2\delta$ for all $n$ and some $0<\delta\leq \pi/2$. Then
    \begin{align*}
        \operatorname{Im}\bra{e^{i\delta}s_n}&\geq \sin\delta,\quad n\in \mathbb{N},
    \end{align*}
    and using $\abs{\paren{J\psi,\psi}}=\abs{e^{i\delta}\paren{J\psi,\psi}}$, we can estimate 
    \begin{align*}
        \abs{\paren{J\psi,\psi}}&\geq  \operatorname{Im}\bra{\sum_{n=1}^{\infty}e^{i\delta}s_n|(\psi,\psi_n)|^2}\geq \sin\delta \sum_{n=1}^{\infty}|(\psi,\psi_n)|^2=\sin\delta\norm{\psi}^2, \quad \psi\in H.
    \end{align*}
    We now verify that $-1$ is not an accumulation point. Suppose to the contrary it is. We define $\varphi_n\in X^{*}$ by
    \begin{align*}
        \varphi_n&\coloneqq \frac{1}{\sqrt{\lambda_n}}A^{*}\psi_n, \quad n\in \mathbb{N},
    \end{align*}
    where the branch of the square root is chosen so that $\operatorname{Im}\bra{\sqrt{\lambda_n}}>0$. Then from the factorization $ATA^{*}\psi_n=F\psi_n=\lambda_n\psi_n$, we observe that
    \begin{align*}
        (T\varphi_n,\varphi_n)&=s_n.
    \end{align*}
    So $T$ satisfies coercivity property and therefore $\bra{\varphi_n}$ is bounded. By the boundedness of $\varphi_n$, without loss of generality we may assume $s_n\to -1$ and $\varphi_n\to \varphi\in X^{*}$ for $n\to \infty$. Then we have 
    \begin{align*}
        (T_0\varphi_n,\varphi_n)+(C\varphi_n,\varphi_n)&=(T\varphi_n,\varphi_n)\to-1,\quad n\to \infty,
    \end{align*}
    but $C$ is compact and so $C\varphi_n\to C\varphi,\quad n\to \infty$. Hence,
    \begin{align*}
        \abs{\paren{C\varphi_n-C\varphi,\varphi_n}}&\leq \norm{C\varphi_n-C\varphi}\norm{\varphi_n}\to 0,\quad n\to \infty,
    \end{align*}
    meaning $\operatorname{Im}\bra{(T\varphi,\varphi)}=\operatorname{Im}\bra{(C\varphi,\varphi)}=0$ and therefore $\varphi=0$ by assumption of the theorem. So we get
    \begin{align*}
        (T_0\varphi_n,\varphi_n)&\to -1,\quad n\to \infty,
    \end{align*}
    contradicting the coercivity of $T$. So $1$ is the only accumulation point of $s_n$, finishing the proof.
\end{proof}
\begin{corollary}
    Let $F$ be the far field operator and assume $k^2$ is not a Dirichlet eigenvalue of $-\Delta$ in $D$. Then
    \begin{align}
        z\in D&\iff \Phi_{\infty}(\cdot,z)\in |F|^{\frac{1}{2}}(H^{\frac{1}{2}}(\partial D))\equiv(F^{*}F)^{\frac{1}{4}}(H^{\frac{1}{2}}(\partial D)).
    \end{align}
\end{corollary}

\begin{corollary}[Picard's Theorem]
    Assume $k^2$ is not a Dirichlet eigenvalue of $-\Delta$ in $D$. Let $(|\lambda_n|,\varphi_n,g_n)$ be a singular system of $F$. Then
    \begin{align}
        z\in D\quad\iff\quad \sum_{n=1}^{\infty}\frac{|(\Phi_{\infty}(\cdot,z),\varphi_n)|^2}{|\lambda_n|}<\infty.
    \end{align}
\end{corollary}
The aid of a singular system $(|\lambda_n|,\varphi_n, g_n)$ of the operator $F$ helps us with the explicit characterization of the scatterer in terms of the range of $|F|^{1/2}=(F^{*}F)^{1/4}$ that can be used for a reconstruction of our scatterer. This is due to the compactness and normality of the far-field operator $F$. We emphasize that the scatterer is not required to be connected. Furthermore, for the application of the factorization method it is not necessary to know whether the scatterer is sound-soft or sound-hard. However, for the impedance boundary condition the far field operator $F$ is no longer normal;  however, the factorization method can modified to rectify this problem. 
\begin{corollary}
    Assume $k^2$ is not a Dirichlet eigenvalue of $-\Delta$ in $D$. For any $z\in \mathbb{R}^2$, the following are equivalent:
    \begin{enumerate}
        \item $z\in D$.
        \item $|F|^{\frac{1}{2}}g_z\equiv (F^{*}F)^{\frac{1}{4}}g_z=\Phi_{\infty}(\cdot,z)$ is solvable in $L^2([0,2\pi])$.
        \item $I(z)\coloneqq \left[\displaystyle\sum_{n=1}^{\infty}\frac{|(\Phi_{\infty}(\cdot,z),\varphi_n)|^2}{|\lambda_n|}\right]^{-1}>0$.
    \end{enumerate}
\end{corollary}
\noindent Again, assuming $k^2$ is not a Dirichlet eigenvalue of $-\Delta$ in $D$, $F$ is normal, one-to-one and compact. Hence, we concluded there exist eigenfunctions $\varphi_n\in \mathbb{C}$ of $F$ that form a complete orthogonal system of $L^2(\Omega)$ with corresponding eigenvalues $\lambda_j\in \mathbb{C}$ of $F$ with $\lambda_n\neq 0$ for $n=1,2,\dots$. We note that $\abs{\lambda_n}$ are the singular values of $F$ and $\bra{|\lambda_n|,\varphi_n,\text{sign}(\lambda_n)\varphi_n}$ is a singular system of $F$. Here we called $\text{sign}(\lambda_n)\varphi_n=g_n$, where $\text{sign}(\lambda_n)=\displaystyle\frac{\lambda_n}{|\lambda_n|}$. By the factorization of $F$, it follows that
\begin{align*}
    -BS^{*}B^{*}\varphi_n&=\lambda_n\varphi_n, \quad n\in\mathbb{N}.
\end{align*}
Define the functions $\bra{\psi_n}\in L^2(\partial D)$ by $B^{*}\varphi_n=-\sqrt{\lambda_n}\psi_n,\, n\in\mathbb{N}$, where we choose a branch of $\sqrt{\lambda_n}$ such that $\operatorname{Im}{\sqrt{\lambda_n}}>0$. Then we obtain
\begin{align*}
    BS^{*}\psi_n&=\sqrt{\lambda_n}\varphi_n,\quad n\in \mathbb{N}.
\end{align*}
It can be shown that the sequence $\bra{\psi_n}$ forms a Riesz basis in the Sobolev space $H^{-\frac{1}{2}}(\partial D)$, i.e., $H^{-\frac{1}{2}}(\partial D)$ consists exactly of those functions $\psi$ of the form
\begin{align*}
    \psi&=\sum_{n=1}^{\infty}\alpha_n\psi_n\quad \text{with }\sum_{n=1}^{\infty}\abs{\alpha_n}^2<\infty.
\end{align*}
Furthermore, as $\bra{\psi_n}$ is a Riesz basis in $H^{-\frac{1}{2}}(\partial D)$, there exists a constant $c>1$ such that
\begin{align*}
    \frac{1}{c^2}\norm{\psi}_{H^{-\frac{1}{2}}(\partial D)}^2\leq \sum_{n=1}^{\infty}\abs{\alpha_n}^2\leq c^2\norm{\psi}_{H^{-\frac{1}{2}}(\partial D)}^2.
\end{align*}
We can now summarize the main theorem for factorization method.
\begin{theorem}
    Assume $k^2$ is not a Dirichlet eigenvalue of $-\Delta$ in $D$. Then the ranges of $B: H^{\frac{1}{2}}(\partial D)\to L^2(\Omega)$ are given by
    \begin{align}
        \mathcal{R}(B)=\bra{\sum_{n=1}^{\infty}\rho_n\varphi_n\,:\, \sum_{n=1}^{\infty} \frac{|\rho_n^z|^2}{|\lambda_n|}<\infty}=\mathcal{R}\paren{\paren{F^{*}F}^{\frac{1}{4}}}.
    \end{align}
    Here $\rho_n^z=\paren{\Phi_{\infty}(\cdot,z),\varphi_n}$ are the expansion coefficients of $\Phi_{\infty}(\cdot, z)$ with respect to $\bra{\varphi_n}$. $\bra{|\lambda_n|,\varphi_n,g_n}$ is the singular system of $F$.
\end{theorem}
This main result yields a complete characterization of our domain $D$.
\begin{corollary}
Assume $k^2$ is not a Dirichlet eigenvalue of $-\Delta$ in $D$. Then
\begin{align}
    D=\bra{z\in \mathbb{R}^2\,:\,\sum_{n=1}^{\infty} \frac{|\rho_n^z|^2}{|\lambda_n|}<\infty}=\bra{z\in \mathbb{R}^2\,:\, \Phi_{\infty}(\cdot,z)\in \mathcal{R}\paren{\paren{F^{*}F}^{\frac{1}{4}}}}
\end{align}
Moreover, there exists a constant $c>1$ such that 
\begin{align}
    \frac{1}{c^2}\norm{\Phi(\cdot,z)}_{H^{-\frac{1}{2}}(\partial D)}^2\leq \sum_{n=1}^{\infty} \frac{|\rho_n^z|^2}{|\lambda_n|}\leq c^2\norm{\Phi(\cdot,z)}_{H^{-\frac{1}{2}}(\partial D)}^2
\end{align}
for all $z\in D$, which describes how the value of the series blows up as $z\in \partial D$. Additionally, $\norm{\Phi(\cdot,z)}_{H^{-\frac{1}{2}}(\partial D)}$ behaves like $\ln{|d(z,\partial D)|}$ in $\mathbb{R}^2$, where $d(z,\partial D)$ denotes the distance of $z\in D$ from the boundary $\partial D$.
\end{corollary}
Our new indicator function is now given by
\begin{align}
    I(z)&=\brak{\sum_{n=1}^{\infty} \frac{|\rho_n^z|^2}{|\lambda_n|}}^{-1}=\brak{\sum_{n=1}^{\infty} \frac{|\paren{\Phi_{\infty}(\cdot,z),\varphi_n}|^2}{|\lambda_n|}}^{-1}.
\end{align}
A major advantage of the factorization method is that we have a way of completely characterizing the scatterer $D$ using $I(z)$, noting 
\begin{align*}
    z\in D&\iff I(z)>0.
\end{align*}
Here we summarize the factorization method as a numerical scheme.
\begin{remark}[Factorization Method]
    The factorization method, like the LSM, provides a scheme for the reconstruction of $D$ by a regularized solution to (2.1) from the scattering data $u_{\infty}(\Hat{x},d)$ for $\Hat{x},d\in \Omega$, where $\Omega$ is the unit circle or sphere for dimensions $2$ and $3$ respectively.
    \begin{enumerate}
        \item Choose a sampling grid $\mathcal{G}$, a regularization parameter $\alpha>0$ and a cut-off constant $c_0$.
        \item For all $z\in \mathcal{G}$ solve the regularized version of the integral equation (2.1) using some appropriate regularization method (usually Tikhonov-Morozov regularization) with parameter $\alpha$.
        \item Calculate a reconstruction $M$ for $D$ by
        \begin{align*}
            M&\coloneqq \bra{z\in \mathcal{G}\, :\, \norm{g_{\alpha,z}}\leq c_0}.
        \end{align*}
        Equivalently, for factorization method in particular, 
        \begin{align*}
            M&\coloneqq \bra{z\in \mathcal{G}\,:\, I(z)>0},
        \end{align*}
        where $I(z)$ is the indicator function given in (2.15). The benefit of the factorization method is that we can calculate $\norm{g_{\alpha,z}}_{L^2(\Omega)}$ directly from the spectral data of $F$. Namely, let $\bra{\sigma_j,\psi_j,\widetilde{\psi}_j}$ be a singular system of $F$. Then we have the representation
        \begin{align*}
            (F^{*}F)^{\frac{1}{4}}g_z&=\sum_{j=1}^{\infty}\sqrt{\sigma_j}\langle g_z,\psi_j\rangle_{L^2}\psi_j.
        \end{align*}
        Thus, the regularized solution has the explicit form
        \begin{align*}
            g_{\alpha,z}&=\sum_{j=1}^{\infty}\frac{\sqrt{\alpha}}{\alpha+\sigma_j}\rho_j^z\psi_j,
        \end{align*}
        and the norm is given by
        \begin{align*}
            \norm{g_{\alpha,z}}^2&=\sum_{j=1}^{\infty}\frac{\alpha}{(\alpha+\sigma_j)^2}\abs{\rho_j^2},
        \end{align*}
        where $\rho_j^2=\langle \Phi_{\infty}(\cdot,z),\psi_j\rangle_{L^2}$.
    \end{enumerate}
\end{remark}

\section{Comparing the Factorization Method to Linear Sampling Method}
The factorization method looks for a solution to the linear equation
\begin{align}
    (F^{*}F)^{\frac{1}{4}}g_z&=\Phi_{\infty}(\cdot,z)
\end{align}
which is ill-posed since $(F^{*}F)^{\frac{1}{4}}:L^2([0,2\pi])\to L^2([0,2\pi])$ is compact. Thus, a regularization scheme is needed to compute the solution to (2.12). Using Tikhonov regularization, a regularization solution $g_{\alpha,z}$ is defined as the solution to the well-posed equation
\begin{align}
    \alpha g_{\alpha,z}+(F^{*}F)^{\frac{1}{2}}g_{\alpha,z}&=(F^{*}F)^{\frac{1}{4}}\Phi_{\infty}(\cdot,z)
\end{align}
where $\alpha>0$ is the regularization parameter that can be chosen according to the Morozov discrepancy principle such that
\begin{align*}
  \norm{(F^{*}F)^{\frac{1}{4}}g_{\alpha,z}-\Phi_{\infty}(\cdot,z)}&=\delta\norm{g_{\alpha,z}}
\end{align*}
with $\delta>0$ referring to the error in the measured far field data. Unlike the far-field equation $Fg_z=\Phi(\cdot,z)$ of the LSM, the modified version (2.12) is in fact solvable if and only if $z\in D$ \cite{cakoni2014qualitative}. Hence, it is possible to obtain a convergence result for the regularized solution of (2.12) when $\delta\to 0$.\\

Unlike the linear sampling method, the factorization method provides a rigorous and exact characterization of the obstacle that is fully explicit and solely based on the measurement operator $F$. The linear sampling method fails to have this feature, since, for points $z$ inside the scatterer, the theoretical framework for the method claims the existence of approximate solutions to the far-field equation. However, how to determine these approximate solutions remains unclear due to the fact that the far field pattern of the fundamental solution is almost never an element of the range of the operator $F$ but the boundary or solution operator $B$ \cite{lechleiter2009factorization}. Recall that the boundary operator $B$ is also called the solution operator since $B$ maps data on the obstacle to the far field pattern of the radiating solution to the Helmholtz equation taking that data from the obstacle, i.e., $B: H^{\frac{1}{2}}(\partial D)\to L^2([0,2\pi]): f\mapsto u_{\infty}$. Though factorization method requires additional assumptions on both the far field operator and its factorizations compared to linear sampling, in return the range of the solution operator coincides with the range of the ``square root" of the far field operator. Thus, for the factorization method, additional structural assumptions are needed to obtain results on range identities for operator factorizations. \\

The factorization method can be seen as a refinement of linear sampling. Both methods use an indicator function to determine whether a point $z$ in a grid is inside or outside the scatterer. Computing the norm of a possible solution $g_z$ to the far-field equation or modified version in factorization for many sampling points $z$ and plotting these norms yields an image of the scatterer. Both algorithms are very efficient compared to other techniques since their numerical implementation requires only a single computation of the singular value decomposition of a discretization of the far field operator \cite{lechleiter2009factorization}. However, the main drawback of both LSM and factorization method is the large amount of data needed for the qualitative inversion procedure. 
\section{Class of Direct Imaging Methods: Probe Methods}
So far we discussed the simplest sampling concepts developed by Colton and Kirsch for the LSM and factorization methods. These sampling methods are called point sampling methods since they provide indicator functions that decide whether the sampling point is inside the interior of the scatterer. Point sampling is not only nice for its simplicity, but these schemes allow us to reconstruct scatterers which consist of an unknown number of separate components or which are not simply connected.\\

In addition to point sampling methods, there are probe methods. Again a special feature of these probe methods, like sampling methods, is that they work even if the boundary condition or physical properties of the scatterer are unknown. In other words, these probe methods do not require \textit{a priori} information of the boundary condition. The basic idea of probe methods is based on the singular behavior of the scattered field of the incident point source on the boundary of the obstacle. We will discuss two such algorithms proposed by Ikehata and Potthast, namely the \textit{probe method} \cite{ikehata1998reconstruction} and the \textit{method of singular sources} \cite{potthast2000stability}. These two methods differ from the previous point sampling methods in that they use different indicator functions that blow up when approaching the boundary of some scatterer. However, the approach to constructing these indicator functions is very distinct from the point sampling methods. Namely, we will take a cone or needle approach to shape reconstruction. The idea is that we will locate the singularity of some point source or singular solution at the tip of some cone or needle. We take an approximation domain, which is chosen as a subset of the complement of the cone or needle. We then use the cone or needle to probe the area under consideration. Both the singular sources method and probe method are commonly based on the behavior of the scattered field $\Phi^s(x,z)$ for incident point sources $\Phi(\cdot,z)$ or higher incident multipoles.
\begin{figure}[!ht]
\centering
\resizebox{1\textwidth}{!}{%
\begin{circuitikz}[scale=0.5]
\tikzstyle{every node}=[font=\normalsize]
\draw [, line width=1.1pt ] (6.75,9.25) ellipse (0.25cm and 1.5cm);
\draw [, line width=0.9pt ] (7.25,9) circle (4cm);
\draw [, line width=1.1pt](11.25,8.5) to[short, -o] (8.25,8.5);
\draw [, line width=1.1pt ] (4.25,11) rectangle (8,7.25);
\draw [, line width=1.1pt ] (4.25,11) rectangle  (8,7.25);
\node [font=\tiny] at (5.35,9.00) {scatterer};
\node [font=\tiny] at (9.5,8.25) {needle};
\node [font=\tiny] at (6.99,11.25) {Approximate Domain G};
\node [font=\tiny] at (7.25,5.80) {outer domain B};
\end{circuitikz}
}%
\caption{The figure shows an approximation domain $G$ containing some elliptic scatterer. On the domain of approximation singular sources are approximated. The needle probes for an approximation to the singular source in the full exterior of the needle. The tip of the needle indicates the location of the singularity. }
\label{The Needle Approach}
\end{figure}
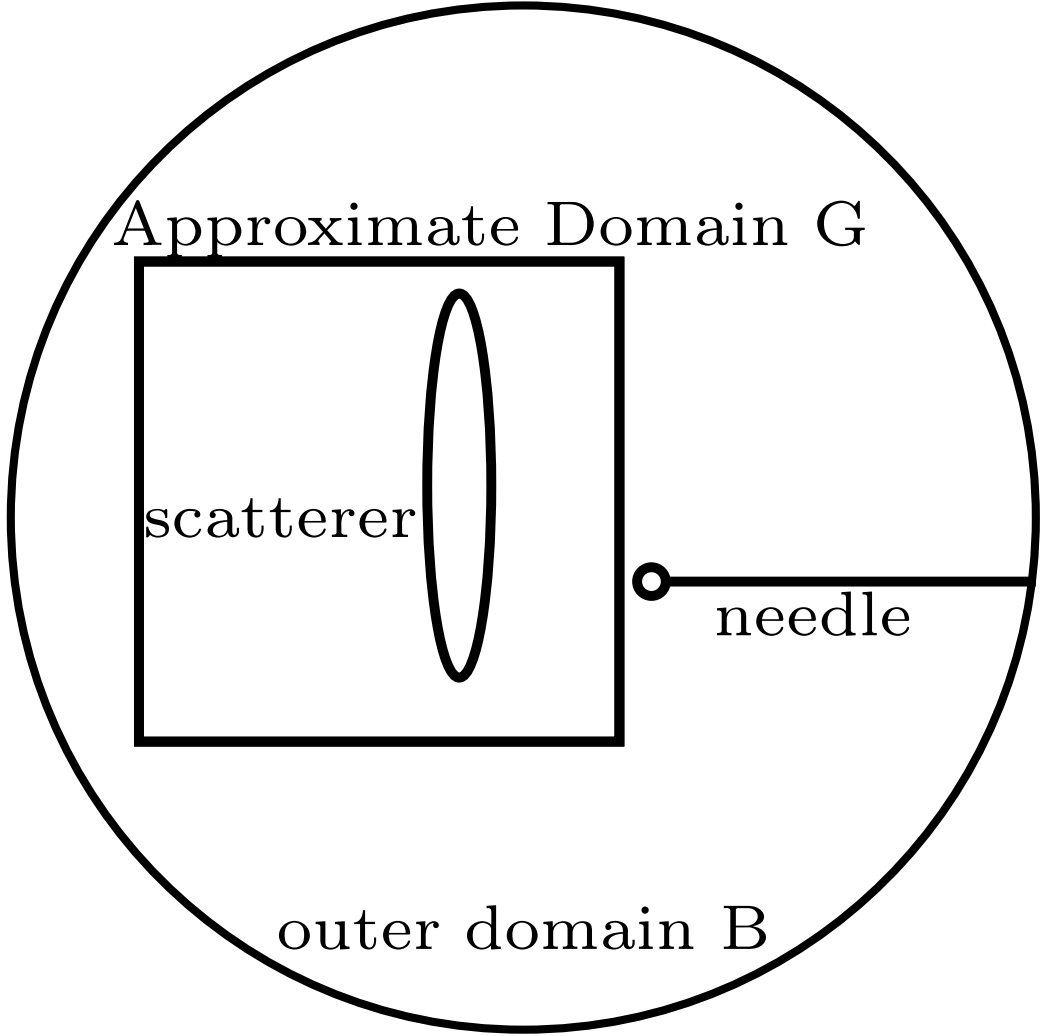
\section{The Method of Singular Sources for Shape Reconstruction}
In the method of singular sources introduced by Potthast, we want to reconstruct the scattered field $\Phi^s(z,z)$ of the singular sources $\Phi(\cdot,z)$ which has singularity in the source point $z$. The idea is to construct an indicator functional defined as the magnitude of the scattered field $\Phi^s(z,z)$ of singular sources $\Phi(z,z)$, which is computed explicitly by backprojection of the form
\begin{align}
    \Phi^s(y,z)&\approx \int_{\Omega}\int_{\Omega}u_{\infty}(\Hat{x},\Hat{y})g(-\Hat{y},z)\,ds(\Hat{y})ds(\Hat{x}),\quad y,z\in \mathbb{R}^m\setminus \overline{D},
\end{align}
for explicitly constructed kernels $g(\cdot,\cdot)$. We will see that this indicator functional also blows up at the boundary of the obstacle if the singularity is chosen appropiately, i.e., 
\begin{align}
    \abs{\Phi^s(z,z)}\to \infty, \quad z\to\partial D
\end{align}
for an acoustic sound-soft or sound-hard scatterer $D$. \\
We again consider the sound-soft acoustic scattering model problem given by (0.1)-(0.3). Assume that the far-field pattern data is retrieved for all $-\Hat{x},d\in \Omega$, where $\Omega$ denotes the unit sphere (or circle for $m=2$). Then the far-field pattern corresponding to a point source $\Phi(\cdot,z)$ has an approximation
\begin{align*}
    \Phi_{\infty}(\Hat{x},z)&\approx \int_{\Omega}u_{\infty}(\Hat{x},d)g_z(d)\,ds(d),\quad \Hat{x}\in \Omega,
\end{align*}
via superposition. Now consider Green's representation formulae for the scattered field and far field pattern of the solution $u$ to (0.1)-(0.3):
\begin{align}
    u^s(x)&=\int_{\partial D}\paren{\Phi(x,y)\frac{\partial u^s}{\partial \nu}(y)-\frac{\partial \Phi(x,y)}{\partial \nu(y)}u^s(y)}\,ds(y),\quad x\in\mathbb{R}^{m}\setminus\overline{D},\\
    u_{\infty}(\Hat{x})&=\gamma\int_{\partial D}\paren{e^{-ik\Hat{x}\cdot y}\frac{\partial u^s}{\partial \nu}(y)-\frac{\partial e^{-ik\Hat{x}\cdot y}}{\partial \nu(y)}u^s(y)}\, ds(y),\quad \Hat{x}\in \Omega,
\end{align}
with the constant
\begin{align*}
    \gamma&\coloneqq\begin{cases}
        \displaystyle\frac{e^{i\pi/4}}{\sqrt{8\pi k}}, \quad m=2\\
        \displaystyle\frac{1}{4\pi},\quad m=3
    \end{cases}
\end{align*}
for radiating solutions of the Helmholtz equation. The method of singular sources is based off Green's formulae, where we have an approximation of the point source $\Phi(\cdot,z)$ at $z$ as an incident field, namely,
\begin{align}
    \Phi(x,z)&\approx \int_{\Omega}e^{ikx\cdot d}g_z(d)\,ds(d),\quad x\in G_z
\end{align}
on some approximation domain $G_z$ with $\overline{D}\subset G_z$ and $z\notin \overline{G_z}$. The estimate is understood in terms of uniformity in $x$ on compact subsets of $G_z$ for fixed $z$. By inserting the approximation of (5.5) into (5.3), and with (5.4) we obtain the approximation
\begin{align}
    u^s(z)&\approx \frac{1}{\gamma}\int_{\Omega}u_{\infty}(-\Hat{x})g_z(\Hat{x})\,ds(\Hat{x})
\end{align}
for the scattered field. This approximation is valid so long as the scatterer $D$ is contained in the approximation domain $G$, i.e., $\overline{D}\subset G_z$. See Figure 4.1. Now we apply (5.6) to the far-field data $\Phi_{\infty}$ and the right-hand side of the previous approximation of this far-field data to obtain
\begin{align}
    \Phi^s(z,z)&\approx \frac{1}{\gamma}\int_{\Omega}\paren{\int_{\Omega}u_{\infty}(-\Hat{x},d)g_z(d)\,ds(d)}\widetilde{g}_z(d)\,ds(\Hat{x}),\quad z\in \mathbb{R}^m\setminus\overline{D},
\end{align}
with some density $\widetilde{g}_z\in L^2(\Omega)$. We will see that the unknown shape is determined by the set of points $z$ where the approximation (5.7) blows up.
The following lemma justifies examining the singular behavior of $\Phi^s(z,z)$ on the boundary.

\begin{lemma}
    If $D\subset \mathbb{R}^2$ is an open set then there exists constants $\tau,c>0$ such that
    \begin{align*}
        \norm{\Phi(\cdot,z)}_{H^{1}(D)}^2&\leq c\abs{\ln d(z,D)}\\
        \norm{\frac{\partial \Phi(\cdot,z)}{\partial \nu}}_{H^{-\frac{1}{2}}(\partial D)}^2&\leq c\abs{\ln d(z,D)}
    \end{align*}
    for every $z\notin D$, which satisfies $0< d(z,D)<\tau$. Furthermore, for every $z\in \mathbb{R}^2\setminus\overline{D}$, we have 
    \begin{align*}
        \norm{\Phi(\cdot,z)}_{H^1(D)}^2&\leq C\abs{\ln d(z,D)}+E,
    \end{align*}
    where the constants $C$ and $E$ depend only on D. 
\end{lemma}
\begin{proof}
    Let $\phi_0(x,z)$ denote the fundamental solution of the Helmholtz equation with wave number $k=0$, i.e., the Laplace equation. We have that $\Phi(x,z)-\Phi_0(x,z)$ is differentiable for all $x,z\in \mathbb{R}^2$, hence $\norm{\Phi(\cdot,z)-\Phi_0(\cdot,z)}_{H^1(D)}^2$ is bounded. Thus it suffices to show that
    \begin{align}
        \norm{\Phi_0(\cdot,z)}_{H^1(D)}^2&\leq C|\ln d(z,D)|+E,
    \end{align}
    for every $z\in \mathbb{R}^2\setminus\overline{D}$. In order to verify this, we recall that we have
    \begin{align*}
        \Phi_0(x,z)&=\frac{1}{2\pi}\ln{\frac{1}{|x-z|}},\\
        \nabla_x\Phi_0(x,z)&=-\frac{1}{2\pi}\frac{x-z}{|x-z|^2}.
    \end{align*}
    Then observe that 
    \begin{align*}
        \norm{\Phi_0(\cdot,z)}_{H^1(D)}^2&\leq C_1\int_D\frac{1}{|x-z|^2}+\paren{\ln{\frac{1}{|x-z|}}}^2\,dx\leq C_2\int_D \frac{1}{|x-z|^2}\,dx\\
        &= C_2 \int_{D\cap B_R(z)}\frac{1}{|x-z|^2}\,dx+C_2\int_{D\setminus B_R(z)}\frac{1}{|x-z|^2}\,dx
    \end{align*}
    where $B_R(z)$ is the ball with the center $z$ and radius $R$. Note that the second integral is clearly bounded because of the boundedness of the scatterer $D$ and the fact that $|x-z|>R$. Additionally, if $d(z,D)=h$, with $z\in \mathbb{R}^2\setminus\overline{D}$ (otherwise the first integral diverges), then for every $x\in D\cap B_R(z)$, we have $h\leq |x-z|\leq R$; therefore, the first integral is bounded from above by 
    \begin{align*}
        C\int_{h}^{R} \frac{2\pi r}{r^2}\,dr &\leq E\ln{\frac{R}{h}}.
    \end{align*}
    So there are constants $C,E>0$ such that, for every $h$,
    \begin{align*}
        \norm{\Phi_0(\cdot,z)}_{H^1(D)}^2&\leq C+E\ln{\frac{1}{h}},
    \end{align*}
    where $h=d(z,D)$. This relation is synonymous with (5.8), as desired. The second inequality in the lemma is similar in verification.
\end{proof}
The next theorem states that our indicator functional, $\abs{\Phi^s(z,z)}$, satisfies the \textit{blow-up property}, i.e., it blows up near the boundary.
\begin{theorem}
    Let $\Phi^s(\cdot,z)$ be the scattered field corresponding to the point source $\Phi(\cdot,z)$. For both the sound-soft and sound-hard boundary conditions, we have
    \begin{align*}
        \lim_{z\to z^{*}}\abs{\Phi^s(z,z)}=\infty,
    \end{align*}
    for all $z^{*}\in \partial D$.
\end{theorem}
This theorem is proven in \cite{fotouhi2005singular}. We can more explicitly state the behavior of the scattered field becoming singular in the source point as follows.
\begin{theorem}
    For the scattering of point sources by a Dirichlet or Neumann scatterer we have the asympotics
    \begin{align}
        \Phi^s(z,z)&=\begin{cases}
            c\ln{d(z,D)}+\mathcal{O}(1),\quad m=2,\\
            \displaystyle\frac{c}{d(z,D)}+\mathcal{O}(\ln{d(z,D)}),\quad m=3,
        \end{cases}
    \end{align}
    for $z\to \partial D$ with some constant $c$.
\end{theorem}
We now summarize the singular sources method by the following steps:
\begin{remark}[Method of Singular Sources]
The method of singular sources is a numerical scheme for shape reconstruction of $D$ by an approximate computation of the function $\Phi^s(z,z)$ from the scattering data $u_{\infty}(\Hat{x},d)$ for $\Hat{x},d\in \Omega$ using the backprojection formula.
\begin{enumerate}
    \item  With a priori knowledge $D\subset B$, where $B$ is a bounded outer domain, choose a domain approximation $G_z$ for each $z\in B$ such that $z\notin \overline{G_z}$ and the unknown inclusion $D\subset G_z$ is valid as far as possible.
    \item Choose value $\tau>0$ and then calculate the density $g_{\tau}(z,\cdot)$.
    \item Choose $\eta$ such that $\eta\norm{g_{\tau}(z,\cdot)}_{L^2(\Omega)}$ becomes sufficiently small. Then calculate $g_{\eta}$.
    \item Calculate the backprojection given in (5.7) and determine the boundary as the set of points where $|\Phi^s(z,z)|$ becomes large. We can define the backprojection operator, $Q$, via
    \begin{align*}
        (Qw)(x,z)&\coloneqq \frac{1}{\gamma}\int_{\Omega}\int_{\Omega}g_{\eta}(x,\Hat{x})g_{\tau}(z,d)w(-d,\Hat{x})\,ds(\Hat{x})\,ds(d).
    \end{align*}
    The backprojection we calculate then is given by $(Qu_{\infty})(z,z)\approx \Phi^s(z,z)$.
\end{enumerate}
\end{remark}
    \section{The Probe Method for Shape Reconstruction}
The probe method was first suggested by Ikehata \cite{ikehata1998reconstruction}. The idea is similar to the singular sources method in that it uses Green's representation formula to define the indicator functional that blows up as the point source approaches the boundary of the unknown obstacle. However, the indicator functional is developed using the Dirichlet-to-Neumann map of the measured data of some physical quantity on the boundary of the scatterer. \\
Let $B$ be a bounded domain in $\mathbb{R}^m$ ($m=2,3$) with Lipschitz boundary. Let $D$ be an open obstacle with Lipschitz boundary of $B$ that satisfies $\overline{D}\subset B$; $B\setminus \overline{D}$
is connected.\\
We denote by $\nu$ the unit outward normal relative to $B\setminus\overline{D}$. Let $k\geq 0$ denote the wave number. Assume $0$ is not a Dirichlet eigenvalue of $\Delta+k^2$ in $B$, nor is $0$ the Dirichlet eigenvalue of the following mixed problem
\begin{align*}
    \Delta u+k^2u&=0\text{ in }B\setminus\overline{D},\\
    \frac{\partial u}{\partial \nu}&=0\text{ on }\partial D,\\
    u&=0\text{ on }\partial B.
\end{align*}
Then, given $f\in H^{\frac{1}{2}}(\partial B)$, let $u\in H^1(B\setminus\overline{D})$ be the weak solution to the mixed elliptic problem
\begin{align*}
    \Delta u+k^2u&=0\text{ in }B\setminus\overline{D},\\
    \frac{\partial u}{\partial \nu}&=0\text{ on }\partial D,\\
    u&=f\text{ on }\partial B.
\end{align*}
Then define
\begin{align}
    \Lambda_Df&=\frac{\partial u}{\partial \nu}\Bigg|_{\partial B}.
\end{align}
Set $\Lambda_D=\Lambda_0$ to the case when $D=\emptyset$. Note that $\Lambda_D: H^{\frac{1}{2}}(\partial D)\to H^{-\frac{1}{2}}(\partial B)$ is called the Dirichlet-to-Neumann (DtN) map.\\
The inverse problem here can be phrased in terms of the DtN map: \begin{problem}
Recover the shape and location of $D$ from $\Lambda_D$ or its partial knowledge.
\end{problem}
Ikehata proposed considering the following indicator function for reconstruction
\begin{align}
    I(z,f)&\coloneqq \int_{\partial B}\overline{\paren{\Lambda_D-\Lambda_0}f}\cdot f\,ds
\end{align}
for specially constructed functions $f$. It can then be shown that $I(z,f)\to \infty$ if $z$ approaches the boundary of the unknown scatterer. The DtN map can be calculated from the far-field patterns $u^{\infty}(\Hat{x},d)$ for all $\Hat{x},d\in \Omega$.\\
We define two notions: a needle and geometric impact parameter (GIP). A continuous curve $c:[0,1]\to \overline{B}$ is called a \textit{needle} if $c(0),c(1)\in \partial B$ and $c(t)\in B$ for all $t\in (0,1)$. Define the geometric impact parameter (GIP) of $c$ with respect to $D$ by 
\begin{align*}
    t(c;D)&\coloneqq \sup\bra{0<t<1\,|\, \forall s\in (0,1), c(s)\in B\setminus\overline{D}},
\end{align*}
Note that if $t(c;D)=1$, then the curve $c([0,1])$ is outside $D$. If $t(c;D)<1$, then the GIP coincides with the hitting parameter of the curve $c$ with respect to the scatterer $D$. The boundary $\partial D$ can then be described as follows
\begin{align}
    \partial D&=\bra{c(t)\,:\, t=T(c;D),\, c\text{ is a needle with }T(c;D)<1}.
\end{align}
Now we proceed to formulate the probe method.
\subsection{Probe Method Formulation}
For any $z\in D$ we can choose a domain of approximation $G_z$ with $z\notin G_z$ and find solutions $v_{n,z}\in H^1(B)$ to the Helmholtz equation in $B$ that approximate the point source $\Phi(\cdot,z)$ on the approximation domain $G_z$, i.e., 
\begin{align*}
    \lim_{n\to\infty}\norm{\Phi(\cdot,z)-v_{n,z}}_{H^1(G_z)}=0.
\end{align*}
Here $v_{n,z}=v_n(z;c_t)$ is a sequence of $H^1(B)$ solutions to the Helmholtz equation with $c_t=\bra{c(s)\,|\,0< s\leq t}$. This is a consequence of the Runge approximation property. (See \cite{ikehata2005new}). Define $f_{n,z}\coloneqq v_{n,z}|_{\partial B}$. Then we can compute the functional (6.2) for any point $z\in B$ with $\overline{D}\subset G_z$, $z\notin G_z$ and calculate the limit
\begin{align}
    \Hat{I}(z)&\coloneqq \lim_{n\to \infty}I(f_{n,z})=\lim_{n\to\infty}\int_{\partial B}\bra{\overline{\paren{\Lambda_D-\Lambda_0}f_{n,z}}}f_{n,z}\,dS,
\end{align}
if it exists. This is the indicator functional. For $z\in D$ the limit $I(z)$ is not defined; however, the functional $I(f_{n,z})$ can be computed. The next theorem establishes the conditions for which the $I(z)$ exists.
\begin{theorem}
    If for some $z\in B$ we have $\overline{D}\subset G_z,z\notin G_z$, then the limit $\Hat{I}(z)$ exists. Furthermore, let $\bra{z_j}$ be a sequence of points for which $\overline{D}\subset G_{z_j}, z_j\notin G_{z_j}$ for $j\in\mathbb{N}$ and $z_j\to \partial D$ for $j\to \infty$. Then we have
    \begin{align*}
        \lim_{j\to\infty}\abs{\operatorname{Re}{\Hat{I}(z_j)}}&=\infty.
    \end{align*}
\end{theorem}
Numerical implementation of the probe method was first obtained by Erhard and Potthast \cite{erhard2006numerical} using ideas arising from point source approximations. These same techniques are additionally important to numerical implementation for the singular source method and no response test.
\subsection{Point Source Approximations}
Consider some point source $z$ and an approximation domain $G_z$. Recall the Herglotz wave operator $H: L^2(\Omega)\to L^2(\partial G_z)$ defined by
\begin{align}
    (Hg)(x)&\coloneqq v_g(x)\Big|_{\partial G_z}=\int_{\Omega} e^{ikx\cdot d}g(d)\,ds(d),\quad x\in \partial G_z.
\end{align}
This operator was an important factor of the factorization of the far field operator $F$. Additionally, $H$ is important for constructing approximations of the point source $\Phi(\cdot,z)$ on $G_z$. Recall that if the only solution to the  homogeneous interior Dirichlet problem for $G(z)$ is the trivial solution, then $H$ is injective and has dense range. Furthermore, it is easy to see $H$ is compact as it is an integral equation with continuous (in fact analytic) kernel. Thus, given some error tolerance $\epsilon$, we can find a density $g\in L^2(\Omega)$ such that
\begin{align}
    \norm{\Phi(\cdot,z)-Hg}_{L^2(\partial G_z)}&\leq \epsilon.
\end{align}
By the well-posedness of the interior Dirichlet problem, for any $M\Subset G_z$ we obtain a constant $c>0$ such that
\begin{align}
    \norm{\Phi(\cdot,z)-v_g}_{C^1(M)}&\leq c\epsilon.
\end{align}
Hence, we have constructed an entire solution to the Helmholtz equation which approximates the point source in the interior of $G_z$. So the density $g$ can be obtained as an approximate solution of the integral equation of the first kind
\begin{align}
    \Phi(\cdot,z)&=Hg\Big|_{\partial G_z}.
\end{align}
But because $H$ is compact, note that equation (6.7), called the \textit{point source equation}, is ill-posed in the sense of Hadamard. So in practice, we need to regularize the equation to compute the density $g$ which approximately solves (6.7). One such efficient scheme is Tikhonov regularization, which computes a stable approximate solution $g_{\alpha}$ by
\begin{align}
    g_{\alpha,z}&\coloneqq (\alpha I+H^{*}H)^{-1}H^{*}\Phi(\cdot,z),
\end{align}
with regularization parameter $\alpha>0$ and the $L^2$-adjoint $H^{*}$ of $H$. Finally, if both the domain $G_z$ and the source point $z$ is translated by some translation vector $t$, then the corresponding solution of the point source equation is obtained by a multiplicative factor 
\begin{align}
    g_{\alpha,(z+t)}(d)&=e^{-ikt\cdot d}g_{\alpha,z}(d),\quad d\in \Omega.
\end{align}
(See \cite{potthast2005sampling} for more details on how the translation property provides a quick scheme to calculate the densities $g_{\alpha,z}$ for a large number of translated approximation domains $G_z$). We can now summarize the scheme for the probe method.
\begin{remark}[Probe Method]
    The probe method is a numerical scheme for shape reconstruction of $D$ by an approximate computation of the Ikehata functional $\Hat{I}$ as defined in (6.3).
    \begin{enumerate}
        \item For each $z\in B$ choose an appropriate domain of approximation $G_z$.
        \item Compute the solution $g_{\alpha_z}$ for the point source equation (6.5).
        \item Compute the boundary value $f_z\coloneqq v_{g_{\alpha}}\Big|_{\partial \Omega}$ of $v$ on $\partial B$.
        \item Compute the Ikehata functional $I(f_z)$ given by (6.2) depending on $z$.
        \item Finally, find the unknown boundary $\partial D$ as the set of points where $|I(f_z)|$ becomes large.
    \end{enumerate}
\end{remark}
\subsection{Choosing the Approximation Domain}
For simplicity, we continue to only consider the two-dimensional case; in principle, we can generalize these results to the three-dimensional case. We firstly consider how to compute approximations for the point sources on the approximated domain $\partial G$ efficiently, since in detecting the obstacle boundary $\partial D$, the approximated domain $\partial G$ needs to be chosen for $z$ approaching $\partial D$ along all directions.\\
For the approximation of a point source $\Phi(\cdot,z)$ on some appropriately chosen approximation domains $G(z)$ by the Herglotz wave operator (6.5), the density $g_{\alpha,z}$ can be calculated as the approximate solution of the point source equation (6.8) via (6.9). Thus, for $z\in B$, it is necessary to solve and regularize an ill-posed equation. This can be made efficient using rigid motions, i.e., translation and rotations. \\
For a fixed reference domain $G_0$ with $0\notin G_0$ and smooth boundary $\partial G_0$, let $G$ be a domain generated from $G_0$ by rotation and translation. Then assume that 
\begin{align}
    G&=\mathbb{M}G_0+z_0,
\end{align}
with a unit orthogonal matrix $\mathbb{M}=(m_{ij})_{2\times 2}$ and translation vector $z_0$. Consider two integral equations of the first kind
\begin{align}
    (Hg_0)(x)&=\Phi(x,0), \quad x\in\partial G_0
\end{align}
and
\begin{align}
    (Hg)(x)&=\Phi(x,z_0),\quad x\in \partial G.
\end{align}
So $G$ emerges from $G_0$ by first rotating the domain $G_0$ and then translating it by the vector $z_0$. Thus, it is sufficient to solve the point source equation only once as stated in the following theorem by Potthast and Erhard.
\begin{theorem}
    Let $\epsilon>0,\,0\in B$ and $g_{\alpha,0}\in L^2(\Omega)$ be the solution to the regularized integral equation
    \begin{align}
        \paren{\alpha I+H^{*}H}g&=H^{*}\Phi(\cdot,0)
    \end{align}
    on $\partial G_0$ such that
    \begin{align}
        \norm{Hg_{\alpha,0}-\Phi(\cdot,0)}_{L^2(\partial G_0)}&<\epsilon.
    \end{align}
    Then for any $z\in B$ with corresponding approximation domain $G_z$ of the form (6.11), the density
    \begin{align}
        g_{\alpha,z}(d)&=e^{-ikx\cdot d}g_{\alpha,0}\paren{\mathbb{M}(z)^td}
    \end{align}
    defines a Herglotz wave function with 
    \begin{align}
        \norm{Hg_{\alpha,z}-\Phi(\cdot,z)}_{L^2(\partial G_z)}&<\epsilon.
    \end{align}
\end{theorem}
The proof of this theorem is fairly straightforward, see \cite{erhard2006numerical}. This theorem reduces the complexity of computing the densities $g_{\alpha,z}$ from solving a complex linear system to pointwise vector multiplication.

In the probe method, we take the near-field data, i.e., the DtN map, simulated by solving the forward mixed problem as the input data. We can take the test domain in $\mathbb{R}^2$ to be
\begin{align*}
    B&=\bra{(x,y)\,:\, x^2+y^2<1}\subset \mathbb{R}^2
\end{align*}
satisfying $\overline{D}\subset B$, and for every point $c(0)\in \partial B$, the straight line needle $c$ connecting $c(0)$ and $0\in D$ has a joint point with $\partial D$. In the algorithm for the probe method, for each $z\in B$ we want to choose an appropriate domain of approximation $G(c,t)$ with $C^2$ regular boundary for a given needle $c$ and point $c(t)\notin \overline{D}$ such that $\bra{c(s)\,:\,0\leq s\leq t}\in B\setminus \overline{G(c,t)}$. But it is enough to construct $G(c_0,t)$ for a special needle $c_0(0)=(0,1)$ since $G(c,t)$ for other needles can be obtained by rotations and a translation.
\begin{remark}[Constructing the Runge Approximation Function]
    We can now specify the scheme for determining the Runge approximation functions $f_n$ on $\partial B$ by constructing the minimum norm solution $g_{\frac{1}{n}}(c(t),d)$ to
    \begin{align*}
        (Hg)(x)&=\Phi(x,c(t)),\quad x\in \partial G(c,t),
    \end{align*}
    with discrepancy $\frac{1}{n}$. For any fixed $c(t)\notin \overline{G(c,t)}$, $g_{\frac{1}{n}}(c(t),d)\coloneqq \phi_0(d)$ can be solved via the equations
    \begin{align*}
    \begin{dcases*}
    \norm{(H\phi_0)(\cdot)-\Phi(\cdot,c(t))}_{L^2(\partial G(c,t))}=\frac{1}{n}\\
    \alpha \phi_0(d)+(H^{*}H\phi_0)(d)=(H^{*}\Phi)(d).
    \end{dcases*}
    \end{align*}
Hence, we can compute $f_n$ explicitly as
\begin{align}
    f_n(x,c(t))&\coloneqq (Hg_{\frac{1}{n}})(x)=\int_{\Omega}e^{ikx\cdot d}g_{\frac{1}{n}}(c(t),d)\, ds(d),\quad x\in \partial B.
\end{align}
We can use this same minimum norm solution as the density for the backprojection in the method of singular sources. (See \cite{liu2011some}).                                      
\end{remark}
\section{Enclosure Method for Polygonal Domains}
All of the sampling and probe methods share the advantage that no knowledge of the boundary condition of the unknown scatterer is needed. Moreover, these methods are still valid for the limited aperture case, i.e., the case where the far field data is not known on the full circle/sphere $\Omega$ but only on an open subset $\Lambda\subset \Omega$. The chief disadvantage of all the sampling and probe methods we covered so far lies in the fact that they all require the knowledge of the far field pattern for a large number of incident waves. However, in practice, such large data is not usually available. The current challenge facing these algorithms is to reduce the amount of data needed for reliable shape reconstruction. Fortunately, there has been some development on reconstruction algorithms using very limited data. One such method is the \textit{enclosure method} by Ikehata (\cite{ikehata1999enclosing} and \cite{ikehata2001enclosure}). The enclosure method is a direct imaging method that enables one to find the support of complex polygons from the knowledge of only one measured field. \\
Ikehata's enclosure method solves the problem of reconstructing a two-dimensional obstacle from Cauchy data on a circle surrounding the obstacle given a total wave field generated by a single incident plane wave with fixed wave number $k$. The obstacle $D$ is a polygonal obstacle, i.e., $D$ takes the form $D_1\cup D_2 \cup \cdots D_m$ with $1\leq m\leq \infty$ where each $D_j$ is open and a polygon; $\overline{D}_j\cap \overline{D}_{j^{'}}=\emptyset$ if $j\neq j^{'}$.\\
The total wave field $u$ is outside the obstacle $D$ and satisfies
\begin{align}
    u(x;d,k)&=e^{ikd\cdot x}+w(x),
\end{align}
with $k>0,\, d\in \Omega$, and satisfies
\begin{align}
    \begin{dcases*}
        \Delta u+k^2u=0\quad \text{in }\mathbb{R}^2\setminus\overline{D},\\
        \frac{\partial u}{\partial \nu}=0\quad \text{on }\partial D\\
        \lim_{r\to\infty}\paren{\frac{\partial w}{\partial r}-ikw}=0,\, r=|x|.
    \end{dcases*}
\end{align}
(7.2) is again the mixed time-harmonic acoustic scattering problem.\\
Let $B_R$ be an open disc with radius $R>0$ centered at a fixed point satisfying $\overline{D}\subset B_R$. Assume $B_R$ is known and our data are $u$ and $\partial_{\nu} u$ on $\partial B_R$. Let $\omega=\begin{pmatrix}
           \omega_{1} \\
           \omega_{2} 
         \end{pmatrix}, \omega^{\perp}=\begin{pmatrix}
           \omega_{2} \\
           -\omega_{1} 
         \end{pmatrix}$ be two unit vectors perpendicular to each other. Set $z=\tau\omega+i\sqrt{\tau^2+k^2}\omega^{\perp}$ with $\tau>0$ a parameter. It is easy to check that $z\cdot z=|z|^2=-k^2$. Then we consider a special complex exponential solution of the Helmholtz equation $(\Delta+k^2)v=0$ in $\mathbb{R}^2$, namely
         \begin{align}
             v_{\tau}(x,\omega)&=\exp{\paren{x\cdot z}}=\exp{\paren{x\cdot (\tau\omega+i\sqrt{\tau^2+k^2}\omega^{\perp}}},\quad x\in \mathbb{R}^2.
         \end{align}
This $v_{\tau}$ not only solves the Helmholtz equation for the ambient space but divides the whole space into two parts: if $x\cdot \omega>t$, then $e^{-\tau t}|v|\to \infty$ as $\tau\to \infty$; if $x\cdot \omega<t$, then $e^{-\tau t}|v|\to 0$ as $\tau\to \infty$. The enclosure method virtually checks whether given $t$ the half plane $x\cdot \omega> t$ touches the unknown obstacle. The main benefit of the enclosure method is that we only need a single incident plane wave to implement it.\\

We now define the \textit{support function} of $D$ to be given by $h_D(\omega)\coloneqq \sup_{x\in D}x\cdot \omega$ and say $\omega$ is \textit{regular} with respect to $D$ if the set $\partial D\cap \bra{x\in \mathbb{R}^2\,:\, x\cdot \omega=h_D(\omega)}$ consists of only one point. Note that from the precise knowledge of $h_D(\omega)$ for all $\omega\in \Omega\coloneqq S^1$ one can obtain the convex hull of the obstacle $D$. The convex hull of $D$ is in fact given by the set $\bigcap_{\omega\in \Omega}\bra{x\in \mathbb{R}^2\,:\,x\cdot\omega<h_D(\omega)}$. The inverse problem can be reformulated as follows
\begin{problem}
Determine the support $h_D(\omega)$ approximately for a given $\omega$ from given measurement data.  
\end{problem} 
From Green's theorem we define the functional to be
\begin{align}
    I(\tau;\omega,d,k)&\coloneqq \int_{\partial B_R} \paren{\frac{\partial u}{\partial\nu}v_{\tau}-\frac{\partial v_{\tau}}{\partial\nu}u}\, dS.
\end{align}
Ikehata showed the following important theorem holds that states that at the corners of the convex polygonal scatterer $D$ this indicator functional becomes unbounded \cite{ikehata1999enclosing}.
\begin{theorem}
    Assume $\omega$ is regular. Then the following limit exists and is valid
    \begin{align}
        \lim_{\tau\to\infty}\frac{1}{\tau}\log{\abs{ \int_{\partial B_R} \paren{\frac{\partial u}{\partial\nu}v_{\tau}-\frac{\partial v_{\tau}}{\partial\nu}u}\, dS}}&=h_{D}\omega.
    \end{align}
    Moreover, we have that
    \begin{enumerate}
        \item if $t\geq h_D(\omega)$, then $\displaystyle\lim_{\tau\to\infty}e^{-\tau t}\abs{I(\tau;\omega,d,k)}=0$.
        \item if $t<h_D(\omega)$, then $\displaystyle\lim_{\tau\to\infty}e^{-\tau t}|I(\tau;\omega,d,k)|=\infty$.
    \end{enumerate}
\end{theorem}
\begin{proof}[Sketch of the proof]
    This proof was developed by Ikehata in \cite{ikehata2010probe}, which we outline here. We begin by introducing a new parameter $\sigma$ instead of $\tau$ defined by $\sigma=\sqrt{\tau^2+k^2}+\tau$. Then we obtain, as $\sigma\to \infty$ the complete asymptotic expansion
    \begin{align}
        \int_{\partial B_R}\paren{\frac{\partial u}{\partial\nu}v(x;z)-\frac{\partial v}{\partial \nu}(x;z)u}\,dS(x)e^{-i\sqrt{\tau^2+k^2}x_0\cdot\omega^{\perp}-\tau h_D(\omega)}&
        \thicksim -i\sum_{n=2}^{\infty}\frac{e^{i\frac{\pi}{2}\lambda_n}k^{\lambda_n}\alpha_n K_n}{\sigma^{\lambda_n}}.
    \end{align}
    Here the $\lambda_n$ refers to the singularity of $u$ at a corner and is given by $\lambda_n=(n-1)\pi/\Theta$, where $\Theta$ is the outside angle of $D$ at $x_0\in \partial D\cap \bra{x\in\mathbb{R}^2\,:\,x\cdot \omega=h_D(\omega)}$. So $\pi<\Theta<2\pi$. $K_n$ are constants depending on $\lambda_n$, $\omega$, and the shape of $D$ around $x_0$. Each $\alpha_n$ for $n\geq 2$ refers to the coefficients of the convergent series expansion of $u$ with polar coefficients at a corner:
    \begin{align*}
        u(r,\theta)&=\alpha_1J_0(kr)+\sum_{n=2}^{\infty}\alpha_n J_{\lambda_n}(kr)\cos{\lambda_n\theta},\quad 0<r\ll1,0<\theta<\Theta.
    \end{align*}
    See \cite{ikehata1999enclosing} for a detailed derivation of this expansion. In other words, let $I_{\omega}(\tau, h_D(\omega))$ denote the left-hand side of (7.6). Then (7.6) tells us that $|I_{\omega}(\tau,h_D(\omega))|$ decays algebraically as $\sigma\to\infty$ and therefore as $\tau\to\infty$. So the statements of Theorem 22 follow from (7.6) and the fact that there exists $ n\geq 2\,:\, \alpha_n K_n$. This latter result can be argued by contradiction. Namely, assume the assertion is not true and that for every $ n\geq 2,\, \alpha_n K_n=0$.\\
    Consider the case when $\Theta/\pi$ is irrational. Then assuming $\forall n\geq 2,\, \alpha_n K_n=0$, we have $\alpha_n=0$ and as such $u(r,\theta)=\alpha_1J_0(kr)$ near a corner. Now this right-hand side is an entire solution to the Helmholtz equation. by unique continuation property we obtain $u(x)=\alpha_1J_0(k\abs{x-x_0}$ in $\mathbb{R}^2\setminus\overline{D}$. But the asymptotic behavior of the right-hand and left-hand sides are different, yielding a contradiction.\\
    Instead consider the case where $\Theta/\pi$ is rational. Then we know that for each $n\geq 2$ with $K_n=0$ the $\lambda_n$ becomes an integer. From the assumption that $\forall n\geq 2,\, \alpha_n K_n=0$, one knows if $n$ satisfies $K_n\neq 0$, then $C_n=0$. Thus we obtained the expansion
    \begin{align*}
        u(r,\theta)&=\sum_{n_j}C_{n_j}J_{\lambda_{n_j}}(kr)\cos{\lambda_{n_j}\theta},
    \end{align*}
    where $n_j\geq 2$ satisfy $K_{n_j}=0$. Since $\lambda_{n_j}$ is an integer and $\lambda_{n_j}\Theta=(n_j-1)\pi$, from the right-hand side of the expansion we obtain that for all $r$ with $0<r\ll 1$ $\displaystyle\frac{\partial u}{\partial \theta(r,\pi)}=\displaystyle\frac{\partial u}{\partial \theta(r,\Theta-\pi)}=0$. Then by a reflection argument \cite{ikehata1999enclosing} one can conclude that this is true for all $r>0$. So $u$ has to be a constant function and with the asymptotic behavior of $\nabla u$ one can conclude that the incident direction $d$ has to be parallel to two linearly independent vectors directed along the lines $\theta=\pi$ and $\theta=\Theta-\pi$, a contradiction. Hence $\exists n\geq 2\,:\, \alpha_n K_n$, and the result follows from the fact that $e^{-\tau t}I(\tau;\omega,d,k)=e^{\tau(h_D(\omega)-t)}I_{\omega}(\tau,h_D(\omega))$ and the algebraic decay shown in (7.6).
\end{proof}
\noindent In summary, we are analyzing the behavior of the field $u$ at the corners and edges of the polygonal domain $D$, i.e., where $u$ is singular. We can determine the minimal positive half plane which contains $D$ in the interior using Theorem 22.
\begin{remark}[The Enclosure Method]
    The enclosure method provides a scheme for the reconstruction of the convex hull of the set of edges of some two-dimensional polygonal scatterer from Cauchy data on an outer disc of radius $R>0$ from the knowledge of only one measured field.
    \begin{enumerate}
        \item Choose a family of half planes $H_j=H[\omega_j,t_j]\coloneqq\bra{x\in\mathbb{R}^2\,: x\cdot w_j<t_j}, j\in \mathcal{J}$, with $\mathcal{J}$ some index set and constant $c_0$.
        \item Choose some regularization parameter $\tau>0$ and define
        \begin{align*}
            v_j(x)&\coloneqq v(\omega_j,\tau,t_j)(x),
        \end{align*}
    as defined in (7.3).
    \item Use the asymptotic behavior of 
    \begin{align*}
        \lim_{\tau\to\infty}\frac{\log{\abs{I(\tau;\omega,d,k)}}}{\tau}&=h_D(\omega),
    \end{align*}
    where $h_D(\omega)=\sup_{x\in D}x\cdot\omega$ to compute some minimal positive half plane $H[\omega, h_D(\omega)]$ which contains $D$ in the interior. Alternatively, for each test half plane $H_j,j\in\mathcal{J}$, determine the indicator functional $I(\tau;\omega_j,d,k)=I_j$ given in (7.4). If $I_j<c_0$, then we call $H_j$ a positive half plane.
    \item The convex hull of the domain $D$ is found to be the intersection of all these half planes for a number of directions $\omega$, i.e., compute
    \begin{align*}
        D_{enc}&\coloneqq \bigcap_{H_j,\text{ positive },j\in\mathcal{J}}H_j,
    \end{align*}
    where each $H_j$ is a positive half plane. $D_{enc}$ will be an approximation to the $\text{conv}(D)$.
    \end{enumerate}
\end{remark}
\section{No Response Test: Another One-Wave Method}
Like the enclosure method, the no-response test introduced by Luke and Potthast \cite{luke2003no} is a related method that reconstructs the shape of a scatterer from the knowledge of the scattered wave or far-field pattern for scattering of only a single time-harmonic incident wave. \\
The inverse problem we consider is to locate and construct the scatterer $D$ given one incident wave $u^i$
and the far field data restricted to the aperture $u_{\infty}|_{\Lambda}$, where $\Lambda\subset \Omega$, i.e., $\Lambda$ is a proper subset of the unit circle. The no response test is a reconstruction algorithm that uses only one incident wave and does not depend on \textit{a priori} information on the scatterer. An additional benefit of the no response test is that it is well suited for limited aperture data. \\
The main idea behind the no response test is to sample by construction special incident fields which are small on some test domain and large outside and then estimate the response to these waves. If the maximum of the sampled responses is small, this indicates that the unknown scatterer is a subset of the test domain. \\
We will apply this method to the sound-soft time-harmonic acoustic scattering problem (0.1)-(0.3). Recall Green's representation formula for the far-field pattern (5.4). Multiply (5.4) by the density $g\in L^2(\Lambda)$ and integrate over $\Lambda\subset\Omega$ to obtain
\begin{align}
    I(g)&\coloneqq \int_{\Lambda}u_{\infty}(-\Hat{x})g(\Hat{x})\,ds(\Hat{x})\\ \nonumber
    &=\frac{1}{\gamma}\int_{\partial D}\int_{\Lambda}\paren{u^s(y)\frac{\partial e^{iky\cdot d}g(d)}{\partial\nu(y)}-\frac{\partial u^s(y)}{\partial\nu(y)}e^{iky\cdot d}g(d)}\,ds(d)ds(y)\\\nonumber
    &=\frac{1}{\gamma}\int_{\partial D}\paren{u^s\frac{\partial v_g}{\partial\nu}-\frac{\partial u^s}{\partial\nu}v_g}\,ds, 
\end{align}
where $v_g$ is again the Herglotz wavefunction. Note we used a reciprocity principle of the far field pattern for the scattering procedure \cite{liu2011some}. Let $v_g$ and its derivatives be small on some fixed test domain $G_0$. Then the above functional $I(g)$ should be small if $D\subset G_0$, while it will be arbitrarily large if $\overline{G}\subset \mathbb{R}^2\setminus\overline{D}$.\\
Let $G_0$ be the admissible test domain and define $I_{\epsilon}$ for $\epsilon>0$ by
\begin{align}
    I_{\epsilon}&\coloneqq \sup{\bra{|I(g)|\,:\,g\in L^2(\Lambda) \text{ such that }\norm{v_g}_{C^1(\overline{G_0})}}\leq \epsilon}.
\end{align}
This defines the \textit{scattering test response}, i.e., the supremum over all the responses for a fixed test domain. The behavior of the scattering test response will help us recover the location and shape of the scatterer. Namely, we define the indicator functional for the no response test via
\begin{align}
    I_0(G_0)&\coloneqq \lim_{\epsilon\to 0}I_{\epsilon}(G_0).
\end{align}
The no response algorithm, like the probe and singular source method, makes use of the template test domain $G_0$ that is rotated and translated around the computational domain. Let $G$ refer to all possible test domains achieved after rotation and/or translation of $G_0$. Then the next theorem shows that we can obtain some upper estimate for the set of singular points $u^s$ by taking the intersections of the sets $G$ for all possible test domains with $I_0(G)=0$.
\begin{theorem}
    Let $G_0$ be a fixed admissible test domain. We have $I_0(G_0)=0$ if the scattered field $u^s$ can be analytically extended into $\mathbb{R}^2\setminus\overline{D}$. If $u^s$ cannot be analytically extended into $\mathbb{R}^2\setminus\overline{D}$, then $I_{\epsilon}(G_0)=\infty$ for all $\epsilon>0$. Hence $I_0(G_0)=\infty$.
\end{theorem}
A proof is given in \cite{potthast2007convergence}. The indicator functional chosen and approach are similar in spirit to the enclosure method. Both the enclosure method and the no response test are one-wave methods. Both only require one incident wave here. This is because to calculate $I_{\epsilon}$ from the far field pattern $u_{\infty}|_{\Lambda}$ for scattering of a plane wave $u^i$ in the direction $-\Hat{x}$, we use the reciprocity relation 
\begin{align}
    u_{\infty}(\Hat{x},-d)&=u_{\infty}(d,-\Hat{x}),\quad \forall \Hat{x},d\in\Lambda\subset\Omega.
\end{align}
Then we obtain 
\begin{align}
    v_{\infty}(\Hat{x})&=\int_{\Lambda}u_{\infty}(\Hat{x},-d)\,g(-d)\,ds(d)\\
    &=\int_{\Lambda}u_{\infty}(d,-\Hat{x})\,g(-d)\,ds(d),
\end{align}
where $v^{\infty}$ is the far field pattern associated to the scattered field generated by $v_g|_{\Lambda}$, i.e., $v_{\infty}=Fg$, where $F: L^2(\Lambda)\to L^2(\Lambda)$ is the far field operator. Thus, from the knowledge of the far field pattern $u_{\infty}(d,-\Hat{x})$, $d\in\Lambda$, for one wave of incidence $-\Hat{x}$, we cam reconstruct $I_{\epsilon}(G)$ for any test domain $G$ by construction of the kernels $g$ of the limited aperture Herglotz wavefunctions. Unlike the enclosure method, the scatterer need not be a convex polygon. The no response test can gather special information about the scatterer. However, we cannot hope to reconstruct the full shape of $D$ by this one-wave method. Instead, a subset of the closure of $D$ built from an approximate set of singular points of the scattered field $u^s$ will be reconstructed. However, there is a multi-wave version of the no-response test that allows for full reconstruction, which is outlined in \cite{liu2011some}.
\begin{remark}[The No Response Test]
    The no-response test by Luke and Potthast reconstructs a subset of scatterer $\overline{D}$ built from an approximate set of singular points of the scattered field, given one incident wave and the far field data restricted to limited aperture $\Lambda\subset \Omega$.
    \begin{enumerate}
        \item Choose a set of test domains $G_j$ for $j$ in some index $\mathcal{J}$ and a sufficiently small parameter $\epsilon$.
        \item For each test domain $G_j$, construct the functions $v_g^{j,l}$ with $g_l\in L^2(\Lambda)$ for $l\in \mathcal{L}$, another index set such that $\norm{v_g^{j,l}}_{L^2(\Lambda)}\leq \epsilon$.
        \item For each $l\in\mathcal{L}$, calculate the indicator functional $I^j(g_l)$ given by (8.1) and take 
        \begin{align*}
            I_{\epsilon}^j&=\sup_{l\in\mathcal{L}}I^j(g_l).
        \end{align*}
        \item Choose a cut-off constant $c_0$. If $\abs{I_{\epsilon}^j(G_j)}\leq c_0$, we call $G_j$ a positive test domain. Then, just as in the enclosure method, we compute the intersection of all the positive test domains, i.e.,
        \begin{align*}
            D_{res}&\coloneqq \bigcap_{G_j \text{ positive } j\in\mathcal{J}}G_j.
        \end{align*}
        This set $D_{res}$ will be an approximation to a subset of $\overline{D}$.
    \end{enumerate}
\end{remark}
\section{Direct Sampling Methods}
We now discuss another point sampling method known as the \textit{orthogonality sampling method}, introduced by Potthast \cite{potthast2010study} and a closely related \textit{direct sampling method} by Liu \cite{liu2016fast}. The key feature of orthogonality/direct sampling is that the computation of the indicator functional involves only the inner products of the measurements, with suitably chosen functions \cite{alzaalig2017direct}. This makes the direct sampling methods especially robust to noise and computationally faster compared to most of the classical methods we outlined. These direct sampling methods additionally work well with limited aperture data. The main idea is to develop an imaging functional using the measured data that are positive in the region you want to recover and are approximately zero outside the region.\\
For simplicity, let $D\subset \mathbb{R}^m$ (for $d=2,3$) be the sound soft scattering obstacle (possibly with multiple components). We assume that the boundary $\partial D$ is a class $\mathcal{C}^2$-smooth closed curve/surface, where exterior $\mathbb{R}^2\setminus\overline{D}$ is connected. The scattered field $u^s(\cdot,z)$ is induced by a point incident field $u^i(\cdot,z)=\Phi(\cdot,z)$, where is the location of the point source. $\Phi(\cdot,z)$ is the fundamental solution of the Helmholtz equation. Therefore, the radiating time-harmonic scattered field $u^s(x,y)\in H_{loc}^1(\mathbb{R^d}\setminus\overline{D}$ given by the point source incident field is the unique solution to the following:
\begin{align*}
    \Delta u^s+k^2u^s&=0\quad \text{in }\mathbb{R}^m\setminus\overline{D}\\
    u^s(\cdot,z)|_{\partial D}&=-\Phi(\cdot,z)\\
    \partial_ru^s-iku^s&=\mathcal{O}\paren{\frac{1}{r^{(m+1)/2}}}\text{ as }r=|x|\to\infty.
\end{align*}
This is the usual sound-soft time-harmonic acoustic scattering problem. Assume $k^2$ is not a Dirichlet eigenvalue for the negative Laplacian in D. We can then assume that we have the measured scattering data $u^s(x,z)$ obtained from solving the direct time-harmonic scattering problem. Recall that $u^s$ satisfies an asymptotic relation from which we obtain the far field pattern $u_{\infty}$. The direct sampling methods devise an indicator functional that will yield a solution to \textit{problem 1} given for the LSM.  The direct sampling methods are based on a suitable factorization for the far-field operator 
\begin{align*}
    &F: L^2(\Omega)\to L^2(\Omega)\\
    Fg(\Hat{x})&=\int_{\Omega}u_{\infty}(\Hat{x},\Hat{y})g(\Hat{y})\,ds(\Hat{y}),\quad \text{where }\Omega=\text{unit sphere/circle}.
\end{align*}
Since the far field pattern is analytic, $F$ is a compact operator. Furthermore, by the assumption on $k^2$ $F$ is additionally injective with dense range. Now the factorization we consider for $F$ is 
\begin{align}
    F&=H^{*}TH.
\end{align}
$H: L^2(\Omega)\to H^{\frac{1}{2}}(\partial D)$ is the usual Herglotz wave operator that maps the kernel $g\in L^2(\Omega)$ to a continuous superposition of plane waves, i.e.,
\begin{align}
    Hg&=v_g|_{\partial D}=\int_{\Omega}e^{ikx\cdot \Hat{y}}g(\Hat{y})\,ds(\Hat{y}),\quad x\in \partial D.
\end{align}
$T=-S^{-1}_{\partial D\to \partial D}$ is the bounded inverse of the single-layer potential operator $S_{\partial D\to \partial D}: H^{-\frac{1}{2}}(\partial D)\to H^{\frac{1}{2}}(\partial D)$ given by
\begin{align}
    S_{\partial D\to \partial D}(\psi(z))&=\int_{\partial D}\Phi(x,z)\psi(z)\,ds(z).
\end{align}
With this factorization the two indicators we study are given by
\begin{align}
    W_1(z)&=\Big(F\varphi_z,\varphi_z\Big)_{L^2(\Omega)}\text{ and }
    W_2(z)=\norm{F\varphi_z}_{L^2(\Omega)},\quad \text{where }\varphi_z=e^{-ikz\cdot \Hat{x}},
\end{align}
where the first indicator functional was introduced by Liu \cite{liu2016fast} and the second indicator functional was introduced by Potthast \cite{potthast2010study}. First, observe that
\begin{align*}
    W_{1}(z)&=\Big(F\varphi_z,\varphi_z\Big)_{L^2(\Omega)}\\
    &= \sup_{\norm{\psi}=1}\abs{\Big(F\varphi_z,\psi\Big)_{L^2(\Omega)}}\\
    &\geq \frac{\abs{\Big(F\varphi_z,\psi\Big)_{L^2(\Omega)}}}{\norm{\psi}}\\
    &\geq C\cdot \abs{\Big(F\varphi_z,\varphi_z\Big)_{L^2(\Omega)}},
\end{align*}
where $C$ is a constant. Additionally, 
\begin{align*}
    \abs{\Big(F\varphi_z,\varphi_z\Big)_{L^2(\Omega)}}&\geq \operatorname{Im}\Big(F\varphi_z,\varphi_z\Big)_{L^2(\Omega)}\\
    &\geq \frac{1}{8\pi}\paren{\frac{k}{2\pi}}^{m-2}\Big(F\varphi_z,F\varphi_z\Big)_{L^2(\Omega)}\\
    &=\frac{1}{8\pi}\paren{\frac{k}{2\pi}}^{m-2}\norm{F\varphi_z}_{L^2(\Omega)}^2,
\end{align*}
a consequence of the following identity for the far field operator
\begin{align}
    F-F^{*}-\frac{1}{8\pi}\paren{\frac{k}{2\pi}}^{m-2}F^{*}F,
\end{align}
where $F^{*}$ is the $L^2$-adjoint of $F$. So we found a constant $C=C(m)>0$ satisfying $W_2\leq CW_1$ for every $z$.
Additionally, $\abs{\Big(F\varphi_z,\varphi_z\Big)_{L^2(\Omega)}}\leq C_0\norm{F\varphi_z}$, where $C_0>0$ is another constant. Hence, both indicator functionals are very similarly related, with 
\begin{align*}
    W_1\lesssim W_2 \text{ and }W_2\lesssim W_1,
\end{align*}
hence, $W_1\simeq W_2$.
We will first justify the efficacy of these functionals. To do this, consider the \textit{Funk-Hecke integral identities} given as follows:
\begin{align}
    \int_{\Omega}e^{-ik(z-x)\cdot \Hat{y}}\,ds(\Hat{y})&=\begin{cases}
        2\pi J_0(k|x-z|)&\text{ if }m=2,\\
        4\pi j_0(k|x-z|)&\text{ if }m=3.
    \end{cases}
\end{align}
Here, $J_0$ is the zeroth Bessel function of the first kind; $j_0$ is zeroth order spherical Bessel function of the first kind. We will make use of the decay of the Bessel functions, i.e.,
\begin{align}
    J_0(t)&=\frac{\cos{t}+\sin{t}}{\sqrt{\pi t}}\bra{1+\mathcal{O}\paren{\frac{1}{t}}}&\text{and }j_0(t)=\frac{\sin{t}}{t}\bra{1+\mathcal{O}\paren{\frac{1}{t}}},
\end{align}
as $t\to \infty$. We then obtain
\begin{lemma}
    \begin{align}
        (H\varphi_z)(x)&=v_{\varphi_z}=\mathcal{O}\paren{\frac{1}{|x-z|^{m-\frac{1}{2}}}}.
    \end{align}
\end{lemma}
This follows immediately from the Funk-Hecke identities.
\begin{theorem}[Image Functional Characterization]
    Let $W_1$ be the image functional as defined in (9.4). Then, for any sampling point $z\in \mathbb{R}^m\setminus\overline{D}$, the following holds:
    \begin{align*}
        W_1(z)&=\mathcal{O}\paren{\text{dist}(z,D)^{1-m}}\text{ as }\text{dist}(z,D)\to\infty.
    \end{align*}
\end{theorem}
\begin{proof}
    Let $v_g$ denote the Herglotz wave function for any $x\in\mathbb{R}^m$. $v_g\in H_{loc}^1(\mathbb{R}^m)$ for any given $g\in L^2(\Omega)$, which implies $v_g|_{\partial D}=Hg\in H^{\frac{1}{2}}(\partial D)$. We see that, for any $x\in \mathbb{R}^m$,
    \begin{align*}
         \abs{W_1(z)}&=\abs{\Big(F\varphi_z,\varphi_z\Big)_{L^2(\Omega)}}\\
         &=\abs{\Big(TH\varphi_z,H\varphi_z\Big)_{L^2(\Omega)}} \qquad(\text{by (9.1)})\\
         &\leq C\norm{H\varphi_z}_{L^2(\partial D)}^2\qquad(\text{since $T$ is bounded})\\
         &=C\norm{v_{\varphi_z}}_{H^{\frac{1}{2}}(\partial D)}^2 \qquad(\text{by definition})\\
         &\leq C\norm{v_{\varphi_z}}_{H^1(D)}^2 \qquad(\text{by the trace theorem}).
    \end{align*}
    Now by lemma 26 we have that
    \begin{align*}
        \norm{v_{\varphi_z}}_{H^1(D)}^2&=\mathcal{O}\paren{\text{dist}(z,D)^{1-m}}\text{ as }\text{dist}(z,D)\to\infty,
    \end{align*}
    where we applied the decay of the Bessel functions.
\end{proof}
What this theorem tells us is that 
\begin{align*}
    W_1(z)\to 0&\text{ as }z\text{ moves far away from the scatterer $D$}.
\end{align*}
Similarly,
\begin{align*}
\norm{F\varphi_z}_{L^2(\Omega)}^2&=\norm{H^{*}TH\varphi_z}_{L^2(\Omega)}^2\\
    &\leq C \norm{H\varphi_z}_{L^2(\partial D)}^2\\
    &= C\norm{v_{\varphi_z}}_{L^2(\partial D)}^2\\
    &\leq C \norm{v_{\varphi_z}}_{H^1(D)}^2\\
    &=\mathcal{O}\paren{\text{dist}(z,D)^{1-m}},\text{ as dist}(z,D)\to\infty.
\end{align*}
\begin{theorem}
    Let $W_2$ be the image functional defined in (9.4). Then, for any sampling point $z\in\mathbb{R}^m\setminus\overline{D}$, the following holds:
    \begin{align*}
        W_2(z)&=\mathcal{O}\paren{\text{dist}(z,D)^{1-m}}\text{ as dist}(z,D)\to\infty.
    \end{align*}
\end{theorem}
The indicator functional $W_1$ was proposed by Liu in \cite{liu2016fast}. We can state it more explicitly as
\begin{align}
    W_1(z)&\coloneqq\abs{\int_{\Omega}e^{-ikz\cdot \Hat{y}}\int_{\Omega}u_{\infty}(\Hat{x},\Hat{y})e^{-ikz\cdot\Hat{x}}\,ds(\Hat{x})ds(\Hat{y})},\quad z\in \mathbb{R}^m.
\end{align}
For the sampling points inside the scatterer $D$, we can always find a lower bound for the indicator $W_1$, and for sampling points away from the scatterer, $W_1$ starts to decay. Moreover, the sulting direct sampling method (DSM) is stable with respect to the noise in the data.
\subsection{Orthogonality Sampling (OSM)}
The indicator functional for orthogonality sampling (OSM) is given by $W_2$, defined as the \textit{reduced scattered field} based on a superposition of Bessel functions. $W_2$ can be rewritten as the inner product between the measurements of the far field pattern and a properly defined test function, namely the far field pattern arising from a point source. That is, for a fixed number $k$ the orthogonality sampling indicator functional is given by
\begin{align}
    W_2(z)&\coloneqq\norm{F\varphi_z}_{L^2(\Omega)}=\abs{\int_{\Omega}e^{-ikz\cdot\Hat{x}}u_{\infty}(\Hat{x})\,ds(\Hat{x})}
\end{align}
on a grid $\mathcal{G}$ of points $z\in \mathbb{R}^m$ ($m=2,3$) from the knowledge of the far field pattern $u_{\infty}$ on $\Omega$. This functional $W_2$ tests the orthogonality of
\begin{align}
    \langle e^{-ikz\cdot\Hat{x}},u_{\infty}\rangle_{L^2(\Omega)},\quad z\in \mathbb{R}^m.
\end{align}
The inner product given by (9.11) determines the orthogonality relation between the measured far field pattern $u_{\infty}$ with a test function which corresponds to the Green's function computed in the far field region. The OSM indicator functional $W_2$ is expected to exhibit large values for sampling points belonging to the targets' support and limited decaying value outside it, as made precise in Theorem 28. As $W_2$ is simply computed as the modulus of the scalar product between the measurements of the far field pattern and a test function without need for regularization, OSM has a significant robustness to noise. The lack of significant robustness to noise is a weakness of the LSM and factorization, so this is one area where OSM (and DSM by Liu) is advantageous. What is especially nice about OSM is that it can be treated as a one-wave method. In other words, OSM is more flexible than LSM and factorization in that OSM does not require data to be acquired under a multi-view multi-static configuration. The indicator functional is still effective even for one incident wave. The main drawback is that only partial shape reconstruction is possible with one incident wave. 
\section{Classifying the Sampling Methods}
We have seen several different sampling concepts emerge with each method. We can classify these sampling methods as follows:
\begin{enumerate}
    \item \textbf{Point Sampling}. These sampling methods are designed to choose a point $z\in \mathbb{R}^m$ (for $m=2,3$)  and construct an indicator function(al) that decides whether a given point $z$ lies inside or outside the scatterer. The benefit of point sampling schemes is that we can construct scatterers which consist of an unknown number of separate components. Furthermore, it is possible to construct scatterers that are not simply-connected. The support of the scatterers is determined by testing the range of special integral operators for a sampling grid containing the domain of interest. The functionals chosen will 'blow up' either outside the scatterer or as the point approaches the boundary of the scatterer. \\
    Examples of point sampling schemes include the \textit{linear sampling method}, \textit{factorization method}, \textit{orthogonality sampling} and \textit{direct sampling methods}. Ideally, we would want to test the range of the far field operator $F$ that gives far field data of the scattered waves. However, as we have seen with the linear sampling method and factorization, we cannot do this directly. In the linear sampling method, it is the range of the data-to-solution operator $B$ that gives us information needed for shape reconstruction. $B$ is a factor of $F$. For the factorization method, we can test the range of the square root of $|F|=(F^{*}F)^{\frac{1}{4}}$, but only with additional structural assumptions on the factorization of $F$ that do not always hold. These methods require a lot of multi-static data to implement. \\
    In comparison to LSM and factorization, the direct/orthogonality sampling methods design an indicator big inside the scatterer and relatively small outside. These methods are simpler to implement since only the inner products of the measurements with some suitably chosen functions are involved in computation of the indicator functional. Direct sampling methods exhibit a greater robustness to noise compared to LSM and factorization. Orthogonality sampling is additionally more flexible, as its application is possible with only one or few incident waves.  \\
    Both OSM and DSM are point sampling methods; however, they are not based on testing the range of an operator based on far field measurements. Additionally, OSM and DSM are more flexible methods in the amount of data needed. However, with only one or few incident waves, shape reconstructions are limited in resolution and quality. \\
    \item \textbf{Probing via a needle or cone}.  The \textit{probe} and \textit{singular source methods} locate the singularity of some point source or singular solution at the tip of a needle or cone. The approximation of a singular solution to the time-harmonic wave equation is only possible on test domains $G$ where the solution is regular and where the singularity is the unbounded component of its complement $R^m\setminus\overline{G}$. Hence, the selection of a test domain $G$, chosen as a subset of the complement of the needle/cone, is crucial to such methods. We then use a needle or cone to probe the area under consideration. The indicator functional chosen determines when the tip of the needle successfully hits the boundary of the scatterer. Such functionals 'blow up' when the tip of the needle touches the boundary of the scatterer. Furthermore, the functional blows up inside the scatterer. The \textit{multi-wave enclosure method} is a special case of the probe method with a special oscillating decaying function $v$ given in (7.3). \\
    \item \textbf{Domain Sampling}. These sampling schemes test whether the desired obstacle lies inside a test domain $G$. Such test domains are called \textit{positive test domains}. By taking the intersection of all such positive test domains, partial information of the scatterer is obtained. Usually, the convex hull of the domain is constructed. These methods overlap with the probing methods in that a subset of the scatterer is built from an approximate set of singular points of the scattered field. The indicator functionals are in fact similar. However, one major benefit of these domain sampling methods is that they are \textit{one wave methods}, i.e., they require only one incident wave and are especially good with limited aperture data. These methods do not require multi-static data. Examples include the \textit{enclosure method} and \textit{no response test}. For the enclosure method, we can alternatively compute the minimal half plane that contains the convex hull of the polygonal scatterer $D$.
\end{enumerate}
\section{Summary}
We have given an overview of several direct imaging methods for inverse time-harmonic acoustic obstacle scattering. Each of these methods recovers basic information on the location, shape, and size of the obstacle. What is crucial to the numerical realization of each method is the choice and calculation of an indicator functional. The indicator functional determines the points that lie inside or outside the scatterer. We have also seen that with some of the methods, the choice of approximate domains or test domains is also key to the accuracy of numerical reconstruction. Additionally, each method involves heuristically selecting a cut-off constant $c_0$. If $c_0$ is chosen to be too large, the reconstruction will be too small to exist. If $c_0$ is chosen to be smaller, the reconstruction may become larger. \\
All the inversion methods involve the regularization of ill-posed integral equations. What makes the integral equations ill-posed is the fact that the operator of interest is compact. In other words, the inversion of compact integral operators like the far field operator or Herglotz wave operator is ill-posed. This is because compact operators cannot have a bounded inverse. Therefore, the choice of a regularization parameter $\alpha$ is also important to direct imaging methods.\\
The choice of indicator functional and the set on which it is defined is a basic criterion to distinguishing different qualitative inversion methods. Some of the indicator functionals are defined on the ambient space $R^m, m=2,3$. Other functionals are defined only on the set of test domains, (e.g., the enclosure method, no response test, and probe method). We focused on the cases where the scatterer is a sound-soft or sound-hard obstacle; however, there is much current work being done on testing these imaging functionals on other scatterers. \\
A special feature of these qualitative schemes is that they usually work if the boundary condition or physical properties of the scatterer are unknown. Furthermore, there is no need for a forward solver or initial guess with such methods. These are the defining features that distinguishes direct imaging methods from both iterative approaches and domain decomposition methods. 
\section{Acknowledgements}
This author is supported in part by NSF grant DMS-2208256.
\nocite{*}
\printbibliography
\end{document}